\documentclass[reqno]{amsart}

\usepackage{amsmath}
\usepackage{amsthm}
\usepackage{amsfonts}
\usepackage{amssymb}
\usepackage{enumerate}
\usepackage{tikz}
\usepackage{tikz-cd}
\usepackage{hyperref}
\usepackage[nocompress]{cite}
\usepackage{mathrsfs}
\usepackage{subcaption}

\DeclareMathOperator{\supp}{supp}
\DeclareMathOperator{\sq}{Sq}

\newcommand{\Prob}{\mathbb{P}}
\newcommand{\E}{\mathbb{E}}
\newcommand{\C}{\mathbb{C}}

\newcommand{\D}{\mathbb{D}}

\newcommand{\eps}{\varepsilon}

\newcommand{\R}{\mathbb{R}}
\newcommand{\N}{\mathbb{N}}

\def\R{\mathbb{R}}

\renewcommand{\d}{\, d }

\theoremstyle{plain}
\newtheorem{theorem}{Theorem}[section]

\newtheorem{lemma}[theorem]{Lemma}
\newtheorem{corollary}[theorem]{Corollary}

\newtheorem{proposition}[theorem]{Proposition}

\theoremstyle{definition}
\newtheorem{definition}[theorem]{Definition}

\newtheorem{assumption}[theorem]{Assumption}

\theoremstyle{remark}
\newtheorem{remark}[theorem]{Remark}

\newcommand{\rdplus}{\oplus}

\newcommand{\editr}{}
\newcommand{\edits}{}

\numberwithin{equation}{section}

\begin{document}
	
	\title[$R$-diagonal elements and random polynomials]{The fractional free convolution of $R$-diagonal elements and random polynomials under repeated differentiation}
	
	\author{Andrew Campbell}
	\address{Institute of Science and Technology Austria, Am Campus 1, 3400 Klosterneuburg, Austria}
	\email{andrew.campbell@ist.ac.at}
	
	\author{Sean O'Rourke}
	\address{Department of Mathematics\\ University of Colorado\\ Campus Box 395\\ Boulder, CO 80309-0395\\USA}
	\email{sean.d.orourke@colorado.edu}
	
	\author{David Renfrew}
	\address{Department of Math and Statistics\\ Binghamton University (SUNY)\\ Binghamton, NY 3902-6000\\USA}
	\email{renfrew@math.binghamton.edu}

	\begin{abstract}
		We extend the free convolution of Brown measures of $R$-diagonal elements introduced by K\"{o}sters and Tikhomirov [Probab. Math. Statist. 38 (2018), no. 2, 359--384] to fractional powers.  We then show how this fractional free convolution arises naturally when studying the roots of random polynomials with independent coefficients under repeated differentiation.  When the proportion of derivatives to the degree approaches one, we establish central limit theorem-type behavior and discuss stable distributions.  
	\end{abstract}

	\maketitle 
	
	\tableofcontents

	\section{Introduction} \label{sec:intro}
	
	The definition of the free convolution $\mu \boxplus \nu$ of two compactly supported probability measures $\mu, \nu$ on the real line is due to Voiculescu \cite{MR839105}.  One can define $\mu \boxplus \nu$ to be the asymptotic limit of the empirical spectral measure of $A_n + B_n$ as $n \to \infty$, where $A_n$ and $B_n$ are independent $n \times n$ random Hermitian matrices, invariant under unitary conjugation, whose individual empirical spectral measures converge to $\mu$ and $\nu$, respectively.  Alternatively, one can define the free convolution using the $R$-transform (see \eqref{eq:RM}, below, for the definition).  Any compactly supported probability measure $\mu$ on the real line is uniquely defined by its $R$-transform $R_{\mu}(z)$ for sufficiently small values of the complex argument $z$.   Heuristically, the $R$-transform can be viewed as the free probability analogue of the cumulant generating function from classical probability theory.  
	
	In fact, the free convolution $\mu \boxplus \nu$ is the unique compactly supported probability measure on the real line whose $R$-transform $R_{\mu \boxplus \nu}$ satisfies 
	\[ {R}_{\mu \boxplus \nu}(z) = {R}_{\mu}(z) + {R}_{\nu}(z) \]
	for all sufficiently small values of $z$.  It follows that for an integer $k \geq 1$, $\mu^{\boxplus k}$, the free convolution of $\mu$ with itself $k$ times, can be characterized by the identity 
	\begin{equation} \label{eq:real:kdef}
		{R}_{\mu^{\boxplus k}}(z) = k {R}_{\mu}(z) 
	\end{equation} 
	for all sufficiently small $z$.  
	In fact, using \eqref{eq:real:kdef}, one can define the fractional free convolution $\mu^{\boxplus k}$ for any real $k \geq 1$.  This was first shown for $k$ sufficiently large by Bercovici and Voiculescu \cite{MR1355057} and then for all real $k \geq 1$ by Nica and Speicher \cite{MR1400060}.

	Amazingly, as the following theorem shows, the free convolution can also be characterized in terms of random polynomials and the roots of their derivatives \cite{MR4669280,doi:10.1080/10586458.2021.1980752,MR4586815}.  We define the {\bf empirical root distribution} of a polynomial $P$ of degree $d$ and roots $x_1, \ldots, x_d$ (counted with multiplicity) to be the probability measure 
	\[ \frac{1}{d} \sum_{i=1}^d \delta_{x_i}, \]
	where $\delta_x$ is the point mass at $x$. 
	
	\begin{theorem}[Hoskins--Kabluchko, Steinerberger, Arizmendi--Garza-Vargas--Perales] \label{thm:AGP}
		Let $\mu$ be a compactly supported probability measure on the real line, and let $P_n$ be the random polynomial 
		\[ P_n(x):= \prod_{i=1}^n (x - X_i), \]
		where $X_1, X_2, \ldots$ are independent and identically distributed (iid) random variables with distribution $\mu$. 
		For any fixed $t \in (0,1)$, the empirical root distribution of the $\lceil tn \rceil$-th derivative of $P_n((1-t)x)$ converges weakly almost surely to $\mu^{\boxplus 1/(1-t)}$ as $n \to \infty$.  
	\end{theorem}
	In other words, for a random polynomial with \edits{iid} real roots, the fractional free convolution $\mu^{\boxplus 1/(1-t)}$ describes the roots of its derivatives.  	The factor of $1-t$ in $P_n((1-t)x)$ simply scales the roots by $(1-t)^{-1}$.  We refer the reader to Section 3 of \cite{MR4586815} for a short proof of a more general version of Theorem \ref{thm:AGP}.  
	
	One of the goals of this paper is to \editr{investigate similar results for a} class of random polynomials with complex roots. In particular, we extend the notion of the free convolution of Brown measures (defined in \eqref{eq:oplus definition}) that was introduced in \cite{MR3896715} to fractional powers and show how these fractional convolution powers can be used to describe a similar relationship as in Theorem \ref{thm:AGP} for random polynomials with independent coefficients.  For a rotationally invariant measure $\mu$ in the complex plane, our result requires acting on $\mu$ by a bijection $\sq$, where $(\sq\mu)(\mathbb{D}_r):=\mu(\mathbb{D}_{\sqrt{r}})$ for $r > 0$ and $\mathbb{D}_r:=\{z \in \mathbb{C}: |z|< r\}$. \editr{An example of how $\sq$ is used to connect the Brown measure with the roots of random polynomials is given in Theorem \ref{thm:A:Kacs free convolution}, below. In fact, this theorem is a special case of our main result, Theorem \ref{thm:A:connection general case}.} We discuss this example more in Section \ref{sec:Haar}.
	
	\begin{theorem}\label{thm:A:Kacs free convolution}
		Let \begin{equation} \label{eq:kacmodel}
			P_n(z):=\sum_{k=0}^{n}\xi_kz^k 
		\end{equation} be a random polynomial such that $\xi_0,\xi_1, \ldots$ are \editr{ iid} standard complex Gaussian random variables. Then the empirical root measure of $P_n$ is known to converge in probability to the uniform probability measure, $\mu$, on the unit circle. 	
		For any fixed $t \in (0,1)$, the empirical root distribution of the $\lceil tn \rceil$-th derivative of $P_n((1-t)^2x)$ converges weakly in probability to $\sq((\sq^{-1}\mu)^{\oplus 1/(1-t)})$ as $n \to \infty$, where the operation $\cdot^{\oplus\editr{1/(1-t)}}$ is a fractional power of a convolution $\oplus$ on rotationally invariant probability measures, defined in Section \ref{sec:fracBrown}.
	\end{theorem}
	Here it is worth noting $\sq^{-1}\mu=\mu$, however we chose to include $\sq^{-1}$ in the statement of Theorem \ref{thm:A:Kacs free convolution} to match the more general result, Theorem \ref{thm:A:connection general case}. It may be helpful to interpret the rescaling of $(1-t)^{-2}$ for the roots of the $\lceil tn \rceil$-th derivative of the polynomial as $(1-t)^{-1}(1-t)^{-1}$, where one factor of $(1-t)^{-1}$ is to account for the natural collapse of the roots under differentiation given by the Gauss--Lucas theorem and the other $(1-t)^{-1}$ factor is to match the diffusion under the convolution.
	
	The definition of $\oplus$ can be technical for those unfamiliar with free probability theory, so we illustrate the connection in Theorem \ref{thm:A:Kacs free convolution} to sums of random matrices, with an example first proved in \cite{MR3091727}. The analogous notion of the empirical root measure for an $n\times n$ random matrix $M$ is the \textbf{empirical spectral measure} $\mu_M$ given by \begin{equation*}
		\mu_{M}:=\frac{1}{n}\sum_{k=1}^n\delta_{\lambda_k(M)},
	\end{equation*} where $\lambda_1(M),\dots,\lambda_n(M) \in \mathbb{C}$ are the eigenvalues of $M$ (counted with algebraic multiplicity). \begin{proposition} [Basak--Dembo] \label{prop:sumsunitary}
		Fix an integer $k \geq 1$, and let $U_n^{(1)},\dots,U_n^{(k)}$ be independent  $n \times n$ Haar distributed unitary random matrices. Then the empirical spectral measure of $U_n^{(1)}+\cdots+U_n^{(k)}$ converges almost surely as $n \to \infty$ to $\mu^{\oplus k}$, where $\mu$ is the uniform probability measure on the unit circle. 
	\end{proposition}
	
	\editr{ The convolution power in Theorem \ref{thm:A:Kacs free convolution} does not need to be an integer, the statement is valid for any power larger than 1, whereas Proposition \ref{prop:sumsunitary} requires integer values. Non-integer powers can be constructed by considering limits of truncations of random matrices (see Sections  \ref{sec:FFCr} and \ref{sec:Frac conv for R}). To illustrate this point, we consider limits of the well studied truncations of Haar distributed unitary random matrices, first computed in \cite{MR1748745,MR2198697}
		\begin{proposition} [ \.Zyczkowski--Sommers, D\'{e}nes--R\'{e}ffy] \label{prop:trunsunitary}
			Fix a $\lambda \in (0,1)$ and let $m :=m_n$ be such that $\frac{m}{n} \to \lambda$ as $n\to\infty$. Let $U_n$ be an $n \times n$ Haar distributed unitary random matrix, and let $P_n^m$ be an $n \times n$ self-adjoint projection, onto an $m$ dimensional subspace. Then the empirical spectral measure of $\frac{n}{m} P_n^m U_n P_n^m$ converges almost surely as $n \to \infty$ to $\mu^{\oplus1/\lambda}$, where as in Theorem \ref{thm:A:Kacs free convolution}, $\mu$ is the uniform probability measure on the unit circle. 
		\end{proposition}
	}
	
	Numerical simulations of Theorem \ref{thm:A:Kacs free convolution} and Proposition \ref{prop:sumsunitary} are given in Figure \ref{fig:radial}. \editr{See Section \ref{sec:Haar} for further details on $\mu^{\oplus k}$.}
	
	The paper is organized as follows.  In Sections \ref{sec:free-prob} and \ref{sec:rand-poly}, we will give some necessary background, known results, and notation concerning free probability theory and random polynomials, respectively. In Section \ref{sec:fracBrown}, we extend the notion of the free convolution of Brown measures (see Section \ref{sec:freeprob}) of $R$-diagonal elements (see Section \ref{sec:Rdiag}) that was introduced in \cite{MR3896715} to fractional powers.  Then in Section \ref{sec:connection}, we will describe how this fractional free convolution is related to roots of derivatives of random polynomials with independent coefficients. In Section \ref{sec:examples} we give several consequences of this theory and state some examples, giving particular attention to the distributions that are stable under $\oplus$ and their relationship to the roots of derivatives of random polynomials.  Finally, in Section \ref{sec:PDE}, we study the dynamics of repeated differentiation using partial differential equations (PDEs), and we relate the PDE limits to our main results.  The appendix contains some auxiliary results.  
	
	\begin{figure}
		\centering
		\begin{subfigure}{.5\textwidth}
			\centering
			\includegraphics[width=0.95\linewidth]{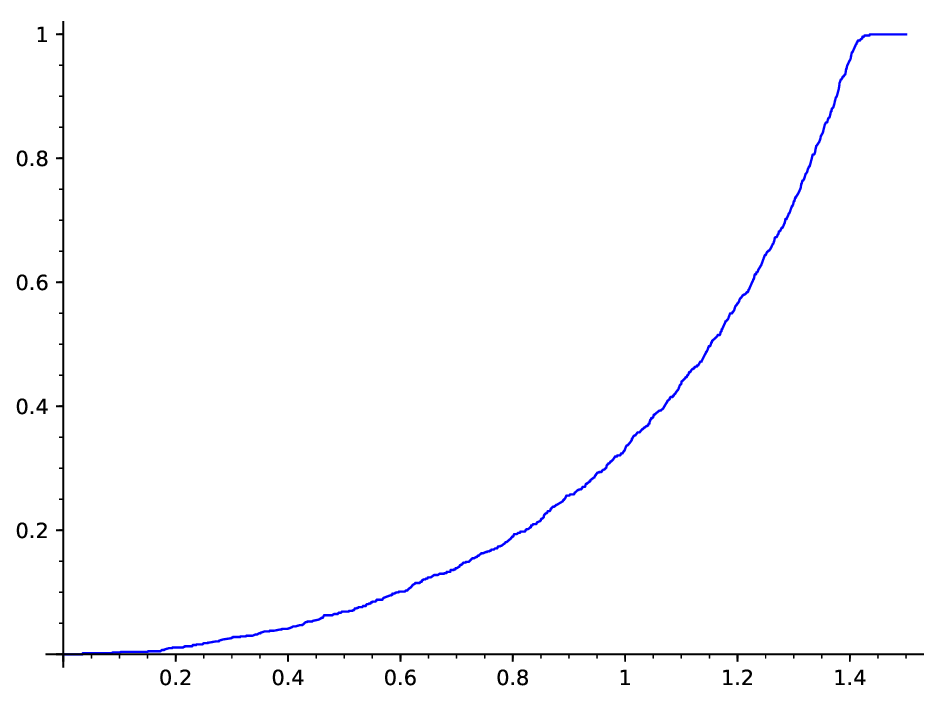}
		\end{subfigure}%
		\begin{subfigure}{.5\textwidth}
			\centering
			\includegraphics[width=0.95\linewidth]{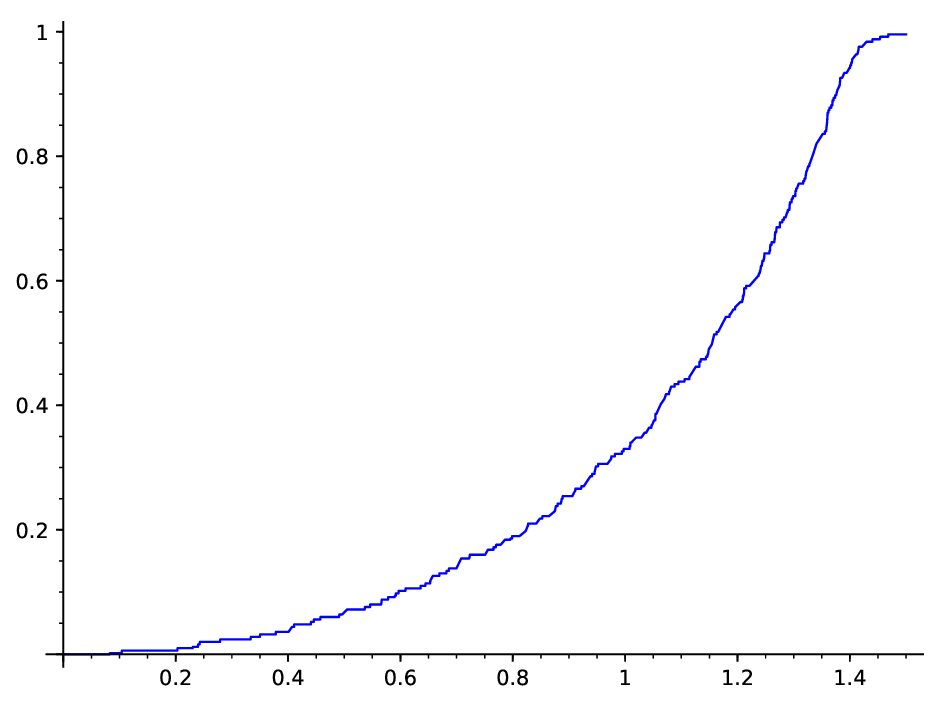} 
		\end{subfigure}
		\caption{Numerical simulations illustrating Theorem \ref{thm:A:Kacs free convolution} and Proposition \ref{prop:sumsunitary}.  The figure on the left shows the radial cumulative distribution function for the eigenvalues of the sum of two independent, Haar distributed unitary random matrices.  The figure on the right is constructed using the random polynomial $P_n$ given in \eqref{eq:kacmodel} when $n = 1000$.  The figure depicts the radial cumulative distribution function of the empirical root measure of the $n/2$-th derivative of $P_n$, after applying the push-forward map $\sq^{-1}$.  }
		\label{fig:radial}
	\end{figure}

	\section{Free \editr{probability theory background}} \label{sec:free-prob}
	
	The large $n$ limit of the empirical spectral measure of $n\times n$ random matrices can often be computed using free probability. In this section, we will introduce the necessary background; we refer the reader to the texts, surveys, and research articles cited in this section for additional details.  
	
	\subsection{Free probability theory background and notation}\label{sec:freeprob}
	
	We work on the non-commutative probability space $(\mathscr{M}, \tau)$, where $\mathscr{M}$ is a von Neumann algebra with normal faithful tracial state $\tau$. When working with unbounded elements we consider the von Neumann algebra that they are affiliated with, see Remarks \ref{rem:aff} and \ref{rem:aff2}, below. An element $u \in \mathscr{M}$ is called \textbf{Haar unitary} if $u^\ast u = u u^\ast = 1$ and $\tau(u^n) = 0$ for all $n \in \mathbb{N}$, where $\mathbb{N} = \{1, 2, 3, \ldots \}$ is the set of natural numbers. Here and in the future we use $1$ to denote the identity operator in $\mathscr{M}$. When $h$ is a self-adjoint element in $\mathscr{M}$, \editr{ the spectral measure, $\mu_h$, is} the unique compactly supported probability measure on $\mathbb{R}$ such that
	\[ \tau(h^n) = \int_{\mathbb{R}} t^n \d \mu_h(t), \qquad n \in \mathbb{N}. \]
	
	We now introduce the free probability transform that we will use to characterize measures. The \textbf{moment generating function} $M_{\mu}$ of a probability measure, $\mu$, on the real line is given by:
	\[M_{\mu}(z) :=  \int_{\mathbb{R}} \frac{zt}{1 - zt} d\mu(t) \]
	for $z \in \C \setminus \supp(\mu)$. Note that if $\mu$ is compactly supported, then for sufficiently small $z$ we have the power series expansion for \editr{$M_{\mu}(z)$}:
	\[\editr{M_{\mu}(z) =} \sum_{k=1}^\infty z^k \tau(h^k) =  \sum_{k=1}^\infty z^k \int_{\mathbb{R}} t^k \d \editr{ \mu}(t) .\]
	
	We then define the \textbf{$R$-transform}\footnote{We note that in the free probability literature, there are two different commonly-used $R$-transforms, which differ by a factor of $z$.} $R_\mu$ of $\mu$ to be the function that satisfies:
	\begin{equation} \label{eq:RM}  R_{\mu}(  z(1 + M_{\mu}(z)) ) =  M_{\mu}(z) ,\end{equation}
	for $z$ in a neighborhood of the origin with $z \neq 0$. In what follows we will identify the various transforms of measures with the corresponding transform of the operator for which the measure was generated, for example $ R_{h} := R_{\mu_h}$.
	
	If $\int_{\mathbb{R}} t \d \mu(t) \neq 0$, we can define the \textbf{$S$-transform}, $\mathscr{S}_\mu$, of $\mu$ as in \cite{MR2266879,MR1217253}, 
	by the identity: 
	\begin{equation} \label{eq:Stran} \mathscr{S}_\mu(z) := \frac{1}{z} R_{\mu}^{\langle -1 \rangle}(z) \end{equation}
	for $z$ in neighborhood of $0$. Here $( \cdot )^{\langle -1 \rangle}$ denotes inversion with respect to composition.
	
	\begin{remark} \label{rem:Stran}
		The primary utility of the $S$-transform is that it linearizes multiplication: if $a$ and $b$ are \editr{self-adjoint and} freely independent \editr{(see \eqref{eq:freeness}, below)} such that $\mu_a$ and $\mu_b$ are supported on $\R^+$, then $\mathscr{S}_{a b}(z) =\mathscr{S}_{a}(z) \mathscr{S}_b(z) $. 
		In what follows it is useful to note that the $S$-transform of the delta mass $\delta_c(x)$ is $\mathscr{S}_{\delta_c}(z) = c^{-1}$ and that the $S$-transform of $\mu_{aa^*}$ can be analytically continued to the open interval $(-1+\mu_a({0}),0)$ and maps this interval monotonically into $\R^+$ (see, for instance \cite{MR1784419}, Theorem 4.4).
		The $S$-transform can alternatively be defined as 
		\[ \mathscr{S}_\mu(z) = \frac{1+z}{z} M_{\mu}^{\langle -1 \rangle}(z), \]
		where the equivalence of these definitions is shown in \cite{MR2266879}, Remarks 16.18 and 18.16. 
	\end{remark}	
	
	When considering a family of (not necessarily self-adjoint) elements  $a_1, \ldots, a_s \in \mathscr{M}$, we consider their \textbf{joint $\ast$-distribution}, given by linear functionals from non-commutative polynomials, $Q(X_1,X_1^*, \ldots, X_s,X_s^*)$, in indeterminants $X_1,X_1^*, \ldots, X_s,X_s^*$ to $\C$: 
	\[ \tau( Q(a_1, a_1^*\ldots, a_s,a_s^*) ) .\]
	When $s=1$, we call this the \textbf{$\ast$-distribution} of $a_1$.

	As in the single variable case, the joint $\ast$-distribution is encoded in the multivariable $R$-transform, $R:=R_{a_1,\ldots,a_s}(z_1,\ldots,z_s),$ which we define to be the power series that satisfies the natural generalization of \eqref{eq:RM}: \[ M = R(z_1(1+M),\ldots,z_n(1+M)),\] where \[ M:=M_{a_1,\ldots,a_k}(z_1,\ldots,z_k) := \sum_{n=1}^\infty \sum_{i_1,\ldots,i_n=1}^k \tau(a_{i_1} \cdots a_{i_n})z_{i_1} \cdots z_{i_n}  ,\]
	is the multivariate moment generating function and $z_1,\ldots,z_k$ are non-commuting indeterminants. The coefficients of the $R$-transform are called the \textbf{free cumulants}, and the $n$-th free cumulant is denoted by $\kappa_n(a_{i_1},\ldots,a_{i_n})$: 
	\begin{equation} \label{eq:multRdef} R(z_1,\ldots,z_k) = \sum_{n=1}^\infty \sum_{i_1,\ldots,i_n=1}^k \kappa_n(a_{i_1},\cdots,a_{i_n})  z_{i_1} \ldots z_{i_n}. \end{equation}
	The free cumulants are multi-linear functions. 
	We refer the reader to \cite{MR2266879}, Section 16, where the free cumulants are instead first defined through a moment-cumulant relation and then Theorem 16.15 and Corollary 16.16 show that this is an equivalent definition. In particular, the free cumulants can be recovered from the joint moments and vice versa.

	We now use the $R$-transform to define free independence of non-commutative random variables. Once again we will give an analytic definition and refer the reader to Section 16  of \cite{MR2266879} for a combinatorial definition in terms of the free cumulants, in particular Theorem 16.6 and Remark 16.7, where the equivalence of the two definitions is shown.

	We say that a collection $a_1, \ldots, a_k$ of elements in $\mathscr{M}$ are \textbf{$\ast$-freely independent} if 
	\begin{equation} \label{eq:freeness} R_{a_1,a_1^* \ldots, a_k,a_k^*}(z_1,z_2 \ldots ,z_{2k-1},z_{2k})  =  R_{a_1,a_1^*}(z_1,z_2)+ \cdots + R_{a_k,a_k^*}(z_{2k-1},z_{2k}) ,
	\end{equation}
	\editr{as formal power series.} In particular, the mixed cumulants vanish. 
	
	\begin{remark} \label{rem:aff}
		When considering an unbounded element $a$, one must instead treat $a$ as an element affiliated to $W^*(a)$, the von Neumann algebra generated by the spectral projections of $|a|$. We then say unbounded elements are freely independent if all elements of their respective affiliated algebras are free. We refer the reader to \cite{MR2339369}, Section 3, for details. 
	\end{remark}
	
	\subsection{The fractional free convolution for self-adjoint operators}\label{sec:FFCr}

	Nica and Speicher \cite{MR1400060} give the fractional convolution powers from \eqref{eq:real:kdef} an additional free probability interpretation, for which we must first introduce additional background. Let $p \in \mathcal{M}$ be a self-adjoint projection \editr{(meaning that $p^2=p$)} with $\tau(p) = \lambda$ for some $\lambda \in (0,1]$, and then consider the new non-commutative probability space $(\mathcal{M}_p, \tau_p)$ given by\footnote{The brackets $[]$ are a formal symbol, which we introduce in order to distinguish $\mathscr{M}_p$ from $\mathscr{M}$.} :
	\[  \mathcal{M}_p := \{  [p a p]: a \in  \mathcal{M} \}  \]
	with 
	\[ \tau_p( [p a p]) = \lambda^{-1} \tau(p a p) \]
	for any $a \in  \mathcal{M}$. We then consider the map $\pi_{\lambda}: \mathcal{M} \to \mathcal{M}_p$ by $ \pi_{\lambda}(a) := [p a p]$, which we will call the {\bf free compression} of $a$. When $\lambda$ is fixed, we will omit it from the notation. Note that in $\mathcal{M}_p$, we have that $ \pi(a^*) = \pi(a)^*, \pi(a) + \pi(b) = \pi(a+b) ,$ and $ \pi(a)\pi(b)= \pi(apb)=\pi(papbp)$.
	\editr{\begin{remark}
			In Corollary 1.14 from \cite{MR1400060}, it is shown that if $a$ is a self-adjoint element in $ \mathcal{M}$ with law $\mu$, that is freely independent of $p$, with $\tau(p)=1/k$, for $k \in \N$, then $k \pi(a )$ has the law $\mu^{\boxplus k}$. In Definition \ref{def:fracconv}, we will give similar convolution semigroup for a class of measures on the complex plane, which from \cite{MR1400060} can also be related to the sum of freely independent elements or the free compression of a single element.
	\end{remark}}
	
	\subsection{The Brown measure}\label{sec:Brown}

	If $a \in \mathscr{M}$ is not \editr{self-adjoint}, then its distribution is not determined by its moments. Nevertheless, there is a distinguished measure associated to $a$, called its Brown measure, which we now introduce.
	We let $\Delta$ denote the Fuglede--Kadison determinant on $(\mathscr{M}, \tau)$ (see \cite{MR52696}), and let $L$ denote $\log \Delta$. It follows that, for $a \in \mathscr{M}$, 
	\[ L(a) = L(a^\ast a) / 2 = L(a^\ast) = \int_{\mathbb{R}} \log t \d \mu_{|a|}(t) \in [-\infty, \infty). \]
	The function $\lambda \mapsto \frac{1}{2 \pi} L(a - \lambda 1)$ is subharmonic on $\mathbb{C}$, and by the Riesz representation theorem can be identified with a regular probability measure, which is called the \textbf{Brown measure} for $a$ (see \cite{MR866489}) and is denoted as $\mu_a$.  The measure $\mu_a$ is defined as
	\[ \mu_a := \frac{1}{2 \pi} \nabla^2 L(a - \lambda 1)  \]
	where $\nabla^2$ denotes the Laplacian, interpreted in the distributional sense. 
	Note that the notation $\mu_a$ agrees with the previously introduced notation for \editr{self-adjoint} elements of $\mathscr{M}$. The Brown measure has a number of important properties  \cite{MR1784419}:
	\begin{itemize}
		\item $\mu_a$ is the unique compactly supported measure that fulfills 
		\[ L(a - \lambda 1) = \int_{\mathbb{C}} \log |z - \lambda| \d \mu_a(z) \]
		for Lebesgue almost all complex numbers $\lambda$.
		\item The support of $\mu_a$ is contained in the spectrum of $a$, and for any natural number $n$
		\[ \tau(a^n) = \int_{\mathbb{C}}z^{\editr{n}} \d \mu_a(z). \]
		\item For any arbitrary $a, b \in \mathscr{M}$, $\mu_{ab} = \mu_{ba}$. 
		\item The Brown measure for a Haar unitary element is the Haar measure on the unit circle.  
		\item The Brown measure of $a$ is determined by its $\ast$-distribution, but it is not continuous with respect to convergence of $\ast$-moments (see, for instance, \cite{MR3585560} Section 11). 
	\end{itemize}

	\subsection{$R$-diagonal elements}\label{sec:Rdiag}
	In general, the Brown measure is difficult to compute, but there is a class of elements, which we now introduce, for which the Brown measure can be computed.
	Let $a$ be an element of $\mathscr{M}$.  Then $a$ is said to be \textbf{$R$-diagonal} if $a$ \editr{has polar decomposition $a = uh$, where $u$ is a Haar unitary free from the radial part $h = |a|$. See Section \editr{15} of \cite{MR2266879} and Sections 2.3 and 2.4 of \cite{MR1876844} for further details about $R$-diagonal elements.}
	%
	Alternatively, an element $a$ is $R$-diagonal if all cumulants except the even cumulants which alternate between $a$ and $a^*$ vanish. We will call such cumulants the \textbf{diagonal terms} of the $R$-transform. In the tracial setting \editr{(meaning the $\tau(ab)=\tau(ba)$ for all $a,b \in \mathscr{M}$)}, the vanishing of the non-diagonal cumulants implies that the $R$-transform of $(a,a^*)$ is of the form \[R_{a,a^*}(z_1,z_2) = \sum_{n=1}^\infty \alpha_n (z_1 z_2)^n + \alpha_n(z_2 z_1)^n,\]  where \[\alpha_n := \kappa_{2n}(a,a^*,\ldots,a,a^*)=\kappa_{2n}(a^*,a,\ldots,a^*,a); \] 
	see \cite{MR2266879}, Example 16.9.

	\begin{remark}
		Classes of bi-unitarily invariant random matrices converge in $\ast$-distribution to $R$-diagonal elements. Although convergence in $\ast$-distribution is not strong enough to guarantee convergence of the empirical spectral measure, it was shown in \cite{MR2831116} that the empirical spectral measure does converge to the Brown measure of an $R$-diagonal element. 
	\end{remark}
	
	The following theorem, from \cite{MR1784419}, see also \cite{2204.01896}, shows that the Brown measure of $R$-diagonal elements can be explicitly computed.  
	
	\begin{theorem}[Haagerup--Larsen, Zhong]
		\label{thm:BM}
		Let $a$ be an $R$-diagonal element, and define 
		\begin{equation} \label{eq:lambda12}
			\lambda_1 := \left(\int_0^{\infty} x^{-2}d\mu_{|a|}(x)\right)^{-1/2}, ~~~  \lambda_2 := \left(\int_0^{\infty} x^2 d\mu_{|a|}(x)\right)^{1/2} ,
		\end{equation}
		with the convention that $\lambda_1 = 0$ if $\int_0^{\infty} x^{-2}d\mu_{|a|}(x) = \infty$.  
		Then the Brown measure of $a$ is radially symmetric and its radial cumulative distribution function (CDF) is given by 
		\begin{equation} \label{eq:Brownmeasure} F_a ( r ): =  \mu_a(\overline{ \D_r}) = \begin{cases} 0 &\text{ if } \editr{ r \in [0,\lambda_1)} \\
				1+\mathscr{S}_{a^*a}^{\langle -1 \rangle}(r^{-2})  &\text{ if } r \in [\lambda_1,\lambda_2)  \\
				1 &\text{ if } r \geq \lambda_2 \end{cases}, \end{equation} 
		where $\overline{ \D_r}$ is the closure of the open disk $\D_r = \{z \in \mathbb{C} : |z| < r \}$. 
	\end{theorem}
	
	\begin{remark}
		Note that the Brown measure is always supported on a (possibly degenerate) ring, centered at the origin. Furthermore, by Remark \ref{rem:Stran}, $\mu_a$ has density when $a$ is not a Haar unitary and $0$ is in its support if $\mathscr{S}_{a^*a}(z)$ is \editr{unbounded}, with its singularity occurring at $z = \mu_a(\{0\})-1$.   				\end{remark}
	
	\begin{remark} \label{rem:aff2}
		If $a$ is unbounded, it is said to be R-diagonal if there exists a von Neumann algebra
		$\mathscr{N}$, with a faithful, normal, tracial state, and $\ast$-free elements $u$ and $h$ affiliated with $\mathscr{N}$, such that $u$ is Haar unitary, $h$ is positive, and $a$ has the same $\ast$-distribution as $u h$. Once again we refer the reader to \cite{MR2339369}, Section 3. The Brown measure of an unbounded operator might not be \editr{compactly} supported, but formula \eqref{eq:Brownmeasure} still holds.
	\end{remark}
	
	Products and sums of freely independent $R$-diagonal elements are also $R$-diagonal \cite{MR1784419}, making the Brown measure of sums of freely independent $R$-diagonal elements a natural object to consider. 				
	From \eqref{eq:Brownmeasure}, we see there is a bijection between the set of Brown measures of $R$-diagonal elements and measures on $\R^+$ such that 
	\[ \int_{\editr{\R^+}} \log^{+}|t| d \nu_{a} < \infty, \]
	given by the correspondence: $\mu_a \leftrightarrow \nu_{a}:=\mu_{a a^*}$.
	Furthermore, there is a bijection between symmetric probability measures on $\R$ and measures on $\R^+$, given by the mapping $x \to  x^2$. 
	We denote by $\widetilde \nu_a$, the inverse image of $\nu_a$ under this map. By composing these two bijections we get a bijection, $\mathcal{H}$, \editr{ mapping a class of symmetric probability measures on $\R$ to the Brown measures of $R$-diagonal elements}. In \cite{MR3896715}, building on the work of \cite{MR1784419}, K\"{o}sters and Tikhomirov show that if $a$ and $b$ are $\ast$-freely independent $R$-diagonal elements, then the Brown measure of $a + b$ is given by\begin{equation}\label{eq:oplus definition}
		\mu_{a+b} = \mu_{a} \rdplus \mu_{b} := \mathcal{H}( \mathcal{H}^{-1}(\mu_a) \boxplus  \mathcal{H}^{-1}(\mu_b) ),
	\end{equation} where $\boxplus$ is the additive free convolution discussed at the beginning of Section \ref{sec:intro}. Note that $\mathcal{H}( \widetilde \nu_a \boxplus  \widetilde \nu_b) = \mathcal{H}( \widetilde \nu_a ) \rdplus \mathcal{H}(  \widetilde \nu_b)$. 
	The convolution $\rdplus$ is used in \cite{MR3896715} to compute the limiting spectrum of a certain class of polynomials of random matrices with iid entries. The authors also characterize the Brown measures that are stable under the $\rdplus$ operation, see Proposition \ref{prop:stablelaw}, below\editr{.}
	
	The Brown measure of an $R$-diagonal element is called \textbf{$\alpha$-$\rdplus$ stable} if for any $m\in\N$
	\[ \mu^{\rdplus m} = \mathcal{D}_{m^{1/\alpha}}\mu, \]
	where, for $c \in (0, \infty)$, $\mathcal{D}_c$ is the scaling operator which maps a probability measure to the measure induced by the mapping $x \to cx$. We also note that in \cite{MR4396250, MR4345335, MR4386403}, a related convolution, denoted $\boxplus_{RD}$, which acts on measures on $\R^+$, was studied.

	\begin{proposition}[K{\"o}sters--Tikhomirov]\label{prop:stablelaw}
		The Brown measure of $a$ is $\alpha$-$\rdplus$ stable if and only if 
		
		\begin{equation} \label{eq:stableS}  \mathscr{S}_{a a^*}(z) = \theta \frac{ (-z)^{\frac{2}{\alpha} -1}}{1+z}\end{equation}
		for some $\theta>0$.
	\end{proposition}
	Here, and throughout this paper, we use the principal branch of the complex function $z^c$ \edits{with branch cut along the negative real axis}. We discuss this proposition further and give an alternative proof in Section \ref{sec:stable}.

	\section{Random \editr{polynomial theory background}}\label{sec:rand-poly}
	
	In this section, we review some results concerning zeros of random polynomials and their derivatives.  We focus on works which are closely related with the results in this paper.  
	
	Let $P_n$ be a (random) polynomial with complex coefficients of degree $n$ in a single complex \editr{variable}.  A natural question is to describe the distribution of the roots of $P_n^{(k)}$, the $k$-th derivative of $P_n$, in terms of the distribution of roots of $P_n$.  In general, the roots of $P_n$ and $P_n^{(k)}$ are related by the Gauss--Lucas theorem, which guarantees the zeros of $P_n^{(k)}$ lie in the convex hull of the roots of $P_n$.  However, the example $P_n(z) = z^n - 1$ shows that the roots of $P_n$ and $P_n^{(k)}$ need not have similar distributions, even when $k = 1$.  
	However, for many models of random polynomials, the roots of $P_n$ and $P_n^{(k)}$ are similar when $n$ tends to infinity and $k$ is fixed (or grows slowly with $n$) \cite{MR3567254,MR3283656,MR3318313,MR3698743,MR2970701,MR4136480,MR3896083,MR4474893,MR3363974,2212.11867,MR3340325,MR3342181,MR3689975}.  
	In this section, we describe some known results for the case when $k$ is proportional to the degree $n$.

	\subsection{Random polynomials with independent coefficients}

	Theorem \ref{thm:AGP} deals with polynomials with independent roots.  A different and more widely-studied model involves polynomials with independent coefficients. \editr{The limiting root measure for these polynomials was described in \cite{MR3262481}}.  Let 
	\begin{equation} \label{eq:pngen}
		P_n(z) := \sum_{k=0}^n \xi_k P_{k,n} z^k 
	\end{equation} 
	be a random polynomial with general coefficients, where $P_{k,n}$ are deterministic \editr{complex} coefficients and $\xi_k$ are non-degenerate iid complex-valued random variables.  It will be convenient to assume that
	\begin{equation} \label{eq:xik}
		\Prob(\xi_0 = 0) = 0 \qquad \text{and} \qquad \E \log(1 + |\xi_0|) < \infty. 
	\end{equation} 
	The coefficients $P_{k,n}$ are assumed to satisfy the following assumption.
	\begin{assumption} \label{assump:a1}
		There exists a function $P:[0, \infty) \to [0, \infty)$ so that 
		\begin{enumerate}[(1)]
			\item\label{assump:a1 p is non neg} $P(t) > 0$ for $t \in [0, 1)$ and $P(t) = 0$ for $t > 1$;
			\item\label{assump:a1 p is continuous} $P$ is continuous on $[0, 1)$ and left continuous at $1$;  and
			\item\label{assump:a1 coefficients converge to p} $\lim_{n \to \infty} \sup_{0 \leq k \leq n} \left| |P_{k,n}|^{1/n} - P( \frac{k}{n}) \right| = 0$. 
		\end{enumerate}
	\end{assumption}
	
	Let $P_n$ be the random polynomial from \eqref{eq:pngen}.  Heuristically, Assumption \ref{assump:a1} implies that the coefficients $P_{k,n}$ are roughly $e^{n \log P(k/n)}$ for some continuous function $P$.  In order to study the roots, we define the random measure 
	\[ \mu_n :=\frac{1}{n} \sum_{z \in \mathbb{C} : P_n(z) = 0} \delta_{z}, \]
	where $\delta_z$ is a Dirac point mass at $z$ and we agree the roots are counted with multiplicities.  
	Recall that for any $r > 0$, $\mathbb{D}_r = \{z \in \mathbb{C} : |z| < r\}$ is the open disk of radius $r$ centered at the origin.  
	In \cite{MR3262481}, Kabluchko and Zaporozhets establish several results describing the asymptotic behavior of the zeros of random analytic functions.  In the special case when the random analytic function is $P_n$, their results reduce to the following.  
	\begin{theorem}[Kabluchko--Zaporozhets  \cite{MR3262481}]\label{thm:KZ}
		Let $P_n$ be the random polynomial given in \eqref{eq:pngen}, where $P_{k,n}$ are deterministic coefficients satisfying Assumption \ref{assump:a1} for some function $P(t)$ and $\xi_0, \xi_1, \ldots$ are iid non-degenerate complex-valued random variables which satisfy $\E \log(1 + |\xi_0|) < \infty$.  Let $I: \mathbb{R} \to \mathbb{R} \cup \{+\infty\}$ be the Legendre-Fenchel transform of $u(t) = -\log P(t)$, where we use the convention that $\log 0 = - \infty$.  That is, 
		\[ I(s) := \sup_{t \geq 0} (st - u(t)) = \sup_{t \geq 0} (st + \log P(t)). \]
		Then $ \mu_n$ converges in probability to the \editr{deterministic, rotationally invariant  measure, $\mu$,  which is characterized by }
		\[ \mu(\mathbb{D}_r) := I'(\log r), \qquad r > 0. \]
	\end{theorem}
	Here, as a convention, $I'$ is the left derivative of $I$.  Since $I$ is convex, the left derivative exists everywhere.  
	
	In \cite{MR3262481}, Kabluchko and Zaporozhets also characterize a set of rotationally invariant measures on $\mathbb{C}$ that arise when one studies the asymptotic behavior of zeros of random analytic functions.  We will need a related class of rotationally invariant probability measures on $\mathbb{C}$.  To this end, we denote by $\mathcal{RP}(\C)$ the set of rotationally invariant probability measures on $\C$ and define the set 
	\[ \mathcal{RP}_p(\C):=\left\{\mu \in \mathcal{RP}(\C)\ :\ \int_{0}^1 \mu\left(\D_r \right)r^{-1}\d r <\infty \right\}. \] 
	We note that the upper bound of $1$ in the integral is not particularly important, and could be replaced by any positive constant for an equivalent definition. \edits{In particular, a measure $\mu \in \mathcal{RP}_p(\C)$ cannot have an atom at the origin.  Here, the subscript `$p$' refers to polynomials since $\mathcal{RP}_p(\C)$ represents the set of probability measures that can arise as the limit of empirical root measures of random polynomials with independent coefficients as explained in Remark \ref{rem:poly} below. }  
	\begin{remark} \label{rem:poly}
		Every measure $\mu\in\mathcal{RP}_p(\C)$ can arise as the limiting empirical root measure of a random polynomial with independent coefficients. This follows from the arguments given by Kabluchko and Zaporozhets in \cite{MR3262481}, Theorem 2.9.  Although Theorem 2.9 from \cite{MR3262481} is stated for random analytic functions, the proof can be specialized to random polynomials when $\mu$ is a probability measure; we now outline the argument. 
		Let $\mu\in\mathcal{RP}_p(\C)$, and define $I(s)=\int_{-\infty}^s\mu(\D_{e^r})dr$. Additionally define the  Legendre--Fenchel transform of $I$:\begin{equation*}
			u(t):=\sup_{s\in\R}(st-I(s)).
		\end{equation*}  Then the random polynomials $P_n(z) = \sum_{k=0}^n \xi_k P_{k,n} z^k$ with $P_{k,n}=e^{-nu(k/n)}$ satisfy Assumption \ref{assump:a1} with $P=e^{-u}$. This follows exactly as in \cite{MR3262481} with the observation that for any finite measure $\mu$ such that $I(s)<\infty$ for all $s\in\R$  one has \begin{equation*}
			\limsup_{t\rightarrow\infty}\frac{I(t)}{t}=\mu(\C),
		\end{equation*} and hence for a probability measure $\mu$, $u(t)=+\infty$ for $t>1$. Thus, $P(t)=0$ for any $t>1$. 
	\end{remark}
	
	Let $P_n$ be the random polynomial from \eqref{eq:pngen}.  We are interested in the $N_n$-th derivative $P_n^{(N_n)}$ of $P_n$, which will be of degree $D_n := n - N_n$. In order to study its zeros, we slightly abuse notation and define the random measure 
	\[ \mu_{D_n} := \frac{1}{D_n}\sum_{z \in \mathbb{C} : P_n^{(N_n)}(z) = 0} \delta_{z}, \]
	where $\delta_z$ is a Dirac point mass at $z$, and we again agree the roots are counted with multiplicities.

	Building on the work of Kabluchko and Zaporozhets \cite{MR3262481}, Feng and Yao \cite{MR3921311} establish the following result for the zeros of $P_n^{(N_n)}$.
	\begin{theorem}[Feng--Yao \cite{MR3921311}] \label{thm:feng-yao}
		Let $P_n$ be the random polynomial given in \eqref{eq:pngen}, where $P_{k,n}$ are deterministic coefficients satisfying Assumption \ref{assump:a1} for some function $P(t)$ and $\xi_0, \xi_1, \ldots$ are iid non-degenerate complex-valued random variables which satisfy \eqref{eq:xik}.  
		\begin{enumerate}[(1)]
			\item\label{thm:part 1:feng-yao} If $\lim_{n \to \infty} N_n/n = 0$, let $I: \mathbb{R} \to \mathbb{R} \cup \{+\infty\}$ be the Legendre-Fenchel transform of $u(x) = -\log P(x)$, then $ \mu_{D_n}$ converges in probability to a rotationally invariant measure $\mu$ in the complex plane given by 
			$ \mu(\mathbb{D}_r) := I'(\log r)$ for all $ r > 0. $
			In particular, $ \mu_{D_n}$ has the same limit as $ \mu_{n}$. 
			\item\label{thm:part 2:feng-yao} If $\lim_{n \to \infty} N_n/n = t \in (0, 1)$, let $ u_t(x) = -\log p(x+t) - (x+t) \log (x+t) + x \log x - (1-t)\log(1-t)$ if $0 \leq x \leq 1-t$ and $-\infty$ if $x > 1 - t$.  Let $I_t: \mathbb{R} \to \mathbb{R} \cup \{+\infty\}$ be the Legendre-Fenchel transform of $ u_t$, then $ \mu_{D_n}$ converges in probability to a rotationally invariant measure $\mu_t$ in the complex plane given by
			$ \mu^{t}(\mathbb{D}_r) := \frac{1}{1-t} I'_t(\log r)$ for all $ r > 0. $
		\end{enumerate}
	\end{theorem}
	
	Here, as a convention, $I'$ and $I_t'$ are the left derivatives of $I$ and $I_t$, respectively.  \edits{The main idea behind Theorem \ref{thm:feng-yao} is that if the coefficients of $P_n$ satisfy Assumption \ref{assump:a1}, then the coefficients of the $N_n$-th derivative of $P_n$ satisfy essentially the same assumption with a (possibly) different exponential profile.  }
	
	In \cite{MR3921311}, Feng and Yao also consider certain special cases, such as the Kac and elliptic models, where they compute the limiting behavior of the zeros when $\lim_{n \to \infty} N_n / n = 1$.  We will discuss these cases and some generalizations in Sections \ref{sec:limit_roots} and \ref{sec:limit of elliptic polynomials}. \editr{ In particular, Proposition \ref{prop:A:limit with n for elliptic} gives a partial answer to questions posed by Feng and Yao on the limiting root distributions when $\lim_{n \to \infty} N_n / n = 1$, illustrating the importance of the tail of the original root measure to the potential limits.}

	\subsection{PDEs describing the behavior of roots under repeated differentiation}
	
	Another approach to studying the distribution of zeros of $P_n$ (or its large $n$ limit) and its $\lceil tn \rceil$-th derivative for some $0 < t < 1$ is to relate them by a partial differential equation (PDE); in this case, we will often think of $t$ as time, with $t = 0$ corresponding to the empirical distribution of roots of $P_n$ (or its large $n$ limit).  
	
	Suppose $P_n$ is a polynomial of degree $n$ having all its roots on the real line with density $f(0, x)$.  In \cite{MR4011508}, Steinerberger introduced the following PDE for the density $f(t,x)$ of the zeros of $P_n^{(\lceil tn \rceil)}$:
	\begin{equation} \label{eq:PDEst} f_t + \frac{1}{\pi} \left( \arctan \left( \frac{ Hf}{f} \right) \right)_x = 0, \end{equation}
	where the equation holds on the support $\supp f$ and $Hf$ is the Hilbert transform of $f$.  
	
	A similar result has been introduced when the roots of $P_n$ are rotationally invariant in the complex plane.  Indeed, given the initial radial density $\psi(x,0)$ of the zeros at $t = 0$, the PDE from \cite{MR4242313} describes the radial density $\psi(x, t)$ at time $0 \leq t < 1$.  The equation is
	\begin{equation} \label{eq:OS}
		\frac{ \partial \psi(x,t) }{\partial t} = \frac{ \partial}{\partial x} \left( \frac{ \psi(x,t) }{ \frac{1}{x} \int_0^x \psi(y,t) dy } \right) \qquad x \geq 0, \quad 0 \leq t < 1. 
	\end{equation}
	Here, we use the convention that $x \geq 0$ either denotes $x \in [0, C]$ (for some finite positive constant $C$) or $x \in [0, \infty)$, depending on whether the density is compactly supported or not. In the former case, by rescaling, we will often assume without loss of generality that $C = 1$.  
	
	In \cite{doi:10.1080/10586458.2021.1980752}, Hoskins and Kabluchko relate the radial (part of the) distribution function 
	\[ \Psi_t(x) := \Psi(x,t) := \int_0^x \psi(y, t) dy \]
	at time $t$ to the initial distribution 
	\[  \Psi_0(x) := \Psi(x, 0) = \int_0^x \psi(y, 0) dy. \]
	They show that $\Psi_t(x)$ satisfies the equation 
	\begin{equation} \label{eq:derivativeflow} \frac{ \Psi_t^{\langle -1 \rangle}(x)}{x}  =  \frac{ \Psi_0^{\langle -1 \rangle}(x+t)}{x+t} \end{equation}
	for $0 < x < 1-t$ and $0 \leq t < 1 $.
	
	In \cite{MR4488834}, Galligo derives a system of two coupled equations to model the motion of real and complex roots for real polynomials under repeated differentiation.  
	
	We explore \eqref{eq:OS} and some related PDEs more in Section \ref{sec:PDE}.  
	
	The functions $\Psi_t$ considered in \eqref{eq:derivativeflow} are radial cumulative distribution functions of sub-probability measures, however a simple normalization results in a similar identity for radial CDFs of probability measures. The function $\Psi_t$ also need not have a true inverse for \eqref{eq:derivativeflow} to provide meaningful information on polynomial roots under differentiation. Instead, \eqref{eq:derivativeflow} can be interpreted as an identity on the generalized left-continuous inverse of the CDF, or \textbf{quantile function} of the distribution. In Appendix \ref{sec:quantile}, we review some basic results on quantile functions which will be used in the proofs of our main results in Sections \ref{sec:Frac conv for R} and \ref{sec:limit_roots}.

	\subsection{Connections to free probability theory}

	The PDE in \eqref{eq:PDEst} also appeared in \cite{MR4530049} to describe the fractional free convolution. In the case of a polynomial with real roots, Steinerberger \cite{MR4669280} proposed an interpretation of the density of zeros of repeated derivatives in terms of free probability theory.  This interpretation has been further explored in \cite{doi:10.1080/10586458.2021.1980752,MR4586815}, culminating in generalized versions of Theorem \ref{thm:AGP}.  In particular, the work \cite{MR4586815} establishes a connection between the real case and finite free probability theory, a subject developed in \cite{MR4408504, 2108.07054}.  
	
	In a similar spirit, Kabluchko \cite{2112.14729} showed that the zeros of real-rooted trigonometric polynomials under repeated differentiation in the asymptotic limit can be described in terms of a free multiplicative convolution involving the free unitary Poisson distribution.  
	
	\editr{While the above works provide connections between differentiation and free probability, and hence random matrices, in the large degree limit these connections can also be seen at finite degree. The following result of Gorin and Marcus \cite{MR4073197} provides such a connection. \begin{theorem}[\cite{MR4073197} Theorem 1.1]\label{prop:Gorin Marcus result}
			Let $P_n$ be a degree $n$ monic polynomial with real roots and let $D_n$ be a $n\times n$ diagonal matrix whose entries are the roots of $P_n$. Let $U_n$ be a $n\times n$ Haar distributed random unitary matrix and $A_{k,n}$ the top left $k\times k$ corner of $UDU^*$. If $Q_n$ is the characteristic polynomial of $A_{k,n}$, then \begin{equation*}
				\E Q_n(z)=\frac{1}{n(n-1)\cdots(k+1)}\left(\frac{\partial}{\partial z}\right)^{n-k}P_n(z).
			\end{equation*} 
		\end{theorem} Taking the $k\times k$ corner of $UDU^*$ can be seen as the finite $n$ version of the free compression described in Section \ref{sec:FFCr}, which is in turn related to free addition. The full result of \cite{MR4073197} is much more general than what is stated here. They consider several random matrix models and operations with natural free probabilistic limits and how they relate to polynomial operations. Similar themes were developed by Gorin and Kleptsyn \cite{GorinKleptsyn2023}.}
	
	In this paper, we further explore connections between zeros of random polynomials and free probability theory in the case when the polynomials have roots in the complex plane.

	\section{The \editr{fractional free convolution for $R$-diagonal elements}}\label{sec:Frac conv for R}
	
	In this section we use the relationship given in \cite{MR1400060} between the distribution of $a$ and $\pi(a)$ to extend the $\rdplus$ operation to fractional powers. We then give an alternative expression for the Brown measure of the sum of $k$ identically distributed, freely independent $R$-diagonal elements. Our expression is more direct, as it does not require using the bijection $\mathcal{H}$\editr{, given in \eqref{eq:oplus definition}}, and computing the free convolution powers of symmetric probability measures.   
	
	\subsection{Fractional free convolution powers of the Brown measure} \label{sec:fracBrown}
	
	\begin{definition}[$R$-diagonal fractional free convolution] \label{def:fracconv}
		Let $a$ be an \editr{$R$}-diagonal element with Brown measure $\mu_a$.  For \editr{$k > 1$} a real number, we define $\mu_a^{\rdplus k}$ to be the radially symmetric probability measure with radial CDF given by 
		\begin{equation} \label{eq:kBrownmeasure} \mu_a^{\rdplus k}(\overline{ \D_r}) := \begin{cases} 
				1 + \mathscr{S}^{\langle -1 \rangle}_k(r^{-2}) &\text{ if } r \in (0,\lambda_2^{(k)})  \\
				1 &\text{ if } r \geq \lambda_2^{(k)} \end{cases} \end{equation} 
		for $r>0$, \editr{where for $z \in \C$, in a neighborhood of $[-1,0]$, we define }
		\begin{equation} \label{eq:Slambda}
			\mathscr{S}_k(z):= \frac{1+ z/k}{k (1+z) } \mathscr{S}_{a a^*}(z/k),   
		\end{equation}
		$\lambda_2^{(k)}: = \sqrt{k} \lambda_2$, and $\lambda_2$ is given by \eqref{eq:lambda12}.
	\end{definition}

	The following proposition gives elementary properties of the probability measure $\mu_a^{\rdplus k}$. Then we will give a proposition that shows the fractional free convolution agrees with the previous definition of $\rdplus$ for integer values of $k$.
	
	\begin{proposition}\label{prop:oplus prop}
		Let $k>1$ be a real number, and let $\mu_a$ be the Brown measure of an $R$-diagonal element $a$. Let $R \in [0,\infty]$ be the \editr{smallest} radius of the \editr{closed} disk that $\mu_a$ is supported on.  Then
		\begin{enumerate}[(1)] 
			\item \label{item:oplus1} $\mu_a^{\rdplus k}(\{0\})  = \max\{0,1-k(1-\mu_a(\{0\}))\} $; 
			\item \label{item:oplus2} on $\C \setminus \{0\}$, $\mu_a^{\rdplus k} $ has density, which is supported and positive on the closed disk of radius $\sqrt{k}R$, centered at the origin. 
		\end{enumerate}
	\end{proposition}

	\begin{proposition} \label{prop:brownmeasuresum} 
		Let $k \geq 1$ be an integer and $a_1, \ldots, a_k$ be freely independent copies of an $R$-diagonal element $a$.  Then the Brown measure of $a_1 + \cdots + a_k$ is $\mu_a^{\rdplus k}$ (as defined in Definition \ref{def:fracconv}).
		Furthermore, $\mu_a^{\rdplus j}$ forms a convolution semigroup:
		\begin{equation} \label{eq:semigroup}
			\left(\mu_a^{\rdplus j}\right)^{\rdplus l} = \mu_a^{\rdplus j l} 
		\end{equation} 
		for all real $j,l \geq 1$.
		
	\end{proposition}

	\begin{proof}[Proof of Proposition \ref{prop:oplus prop}]
		Let $\nu$ be the spectral measure of $a a^*$ and $\mathscr{S}_k$ be as in \eqref{eq:Slambda}. \editr{Throughout the proof we will use that} $\mathscr{S}_{\nu}$ is a decreasing function on $(-1+\nu(\{0\}),0)$ with range $(\lambda_2^{-2},\lambda_1^{-2})$.
		
		To prove \ref{item:oplus1}, we first note that if $\mu_a(\{0\})=0$ then $\mathscr{S}_\nu$ and hence $\mathscr{S}_k$ is finite on $(-1,0)$ and thus $\mu^{\rdplus k}(\{0\}) ~=~0$. If $\mu_a({0})\not = 0$ then $\mathscr{S}_{\nu}$ is singular at $-1+\mu_a(\{0\})$, and thus by \eqref{eq:Slambda}, $\mathscr{S}_k$ is singular at $k(-1+\mu_a(\{0\}))$, \editr{ giving an atom at zero with weight $1-k(1-\mu_a(\{0\}))$ if this quantity is positive and no atom otherwise.}

		To prove \ref{item:oplus2}, we use that the prefactor, $\frac{1 + z/k}{k(1+z)}$, and hence the entire term in \eqref{eq:Slambda} is strictly monotonic \editr{on $[-1,0]$}.  Thus, when combined with \eqref{eq:Brownmeasure}, we find that $\mu_a^{\rdplus k}$ has positive density (by Corollary 4.5 of \cite{MR1784419} the density of the Brown measure is positive on its support). Then showing that the inner radius of $\mu_a^{\rdplus k}$ is $0$ is equivalent to showing that the $S$-transform is singular at $-1~+~\mu_a^{\rdplus k}(\{0\})$. If $\mu_a^{\rdplus k}(\{0\})> 0$, from \editr{\ref{item:oplus1}} we see that this happens. On the other hand, if $\mu_a^{\rdplus k}(\{0\})~=~0$, then from \eqref{eq:Slambda} we have that for all $k>1$, the prefactor is singular at $-1$. To compute the outer radius, we note that $\mathscr{S}_k(0) = \frac{1}{k} \mathscr{S}_{aa^*}(0) = \frac{1}{k R^2}$, and conclude that $\mu_a^{\rdplus k}$ is supported on the disk of radius $\sqrt{k} R$, as desired.
	\end{proof}

	Before we prove Proposition \ref{prop:brownmeasuresum}, we will show that the $\ast$-distributions of $k \pi_{k^{-1}}(a)$ and $a_1 + \cdots + a_k$ are the same, and hence both elements have the same Brown measure. Proposition \ref{prop:brownmeasuresum} will then follow by computing the Brown measure of $\pi(a)$. \editr{As mentioned in Section \ref{sec:FFCr}, \cite{MR1400060} provides a natural connection between $\pi_{k^{-1}}(a)$ and $\sum_{i=1}^k a_i$. In fact, in \cite{MR1400060} it is shown that this relationship holds for not just for a single $a$ but for the joint distribution of the family $(a_1, \ldots, a_k)$. The following lemma follows from their results (see also \cite{MR1400060} Sections 14 and 15), we include it for completeness. }
	
	\begin{lemma} \label{lem:brownequiv}
		Let $a, a_1, \ldots, a_k$  be as in Proposition \ref{prop:brownmeasuresum}. The $\ast$-distributions of $k \pi_{k^{-1}}(a)$ and $a_1 + \cdots + a_k$ are equal. Furthermore, they are both $R$-diagonal elements.
	\end{lemma}
	
	\begin{proof}
		
		We begin by relating the free cumulants of $(\pi_\lambda(a), \pi_\lambda(a)^*)$ to those of $(a,a^*)$. We then specialize to $\lambda = 1/k$. We will now omit the subscript $\lambda$ from $\pi_{\lambda}$.
		
		By Theorem 14.10 of \cite{MR2266879} the free cumulants of $(\pi(a), \pi( a)^*) $ are a rescaling of the free cumulants $(a,a^*)$ by $\lambda^{-1}$:
		\begin{equation}\label{eq:projfreecumulant} \kappa_n^{\mathscr{M}_p}( \pi( a^{\eps_1}),\ldots ,\pi(a^{\eps_n}) ) = \lambda^{-1} \kappa_n(\lambda a^{\eps_1},\ldots,\lambda a^{\eps_n}) \end{equation}
		with $\eps_i \in \{1,* \}$. Here we have introduced the superscript $\mathscr{M}_p$ to make it clear that all relevant quantities are computed with respect to $\tau_p$.
		In particular, because $a$ is $R$-diagonal, the non-diagonal cumulants vanish, meaning $\pi(\lambda^{-1} a)$ is also $R$-diagonal, and its diagonal cumulants are:
		\[ \kappa_n^{\mathscr{M}_p}(\pi(\lambda^{-1} a^{}), \pi(\lambda^{-1} a^{*})\ldots ,\pi  (\lambda^{-1}a^{}), \pi(\lambda^{-1}a^{*})) =  \lambda^{-1} \kappa_{n}(a, a^*,\ldots, a, a^*). \]
		On the other hand, if $a_1, \ldots, a_k$ are $\ast$-freely independent, non-commutative random variables with the same $\ast$-distribution as $a$, then $x^{(k)} := a_1 +  \cdots + a_k$ is also $R$-diagonal with
		\[ \kappa_n(x^{(k)}, x^{(k)*}, \ldots, x^{(k)}, x^{(k)*}) = \sum_{i=1}^k \kappa_n(a_i,a_i^*,\ldots, a_i, a_i^*) = k \, \kappa_n(a,a^*,\ldots, a, a^*) , \]
		where we have used that, by freeness, the mixed cumulants vanish.
		
		Setting $\lambda = k^{-1}$, we see that $x^{(k)}$ and $k \pi(a)$ have the same $\ast$-distribution, as desired.		
	\end{proof}

	To compute the Brown measure of $a_1+\cdots+a_k$ in Proposition \ref{prop:brownmeasuresum}, it now suffices to compute the $S$-transform of $ \pi_{\editr{k^{-1} }}(a) \pi_{\editr{k^{-1} } }(a)^*$; we will use the following lemma to compute this $S$-transform. We remark that similar computations were done in \cite{MR4085372} to study the Brown measure of products of truncations of $\ast$-freely independent Haar unitary elements.
	
	\begin{lemma} \label{lem:Spi} Let $p \in \mathscr{M}$ be a projection with $\tau(p) = \lambda \in (0,1]$, $a \in \mathscr{M}$, \editr{an $R$-diagonal element}, \footnote{\editr{We thank Martin Auer and an anonymous referee for pointing out an error in our previous assumptions to this lemma.}} and $ x \in \mathscr{M}$ be self-adjoint, such that $p$ is free from $a$ and $x$. Then we have for $z$ in a neighborhood of the origin that:
		\begin{enumerate}[(1)]
			\item $\mathscr{S}_p(z) = \frac{1+ z}{\lambda + z} $ \label{Spi1}
			\item $\mathscr{S}_{p a p a^*p}(z) = \left( \frac{1+ z}{\lambda + z} \right)^2 \mathscr{S}_{ a  a^*}(z)  $ \label{Spi2}
			\item $\mathscr{S}_{\pi(x)}(z) =  \frac{\lambda(z+1)}{\lambda z+1} \mathscr{S}_{pxp}(\lambda z)$ \label{Spi3}
		\end{enumerate}
	\end{lemma}
	\begin{proof}
		To prove \ref{Spi1}, we begin by computing the moment generating function for $p$:
		\[ M_p(z) = \sum_{k=1}^\infty \tau(p^k) z^k = \sum_{k=1}^\infty \lambda z^k = \frac{ \lambda z}{ 1 - z}. \]  
		We then compute its inverse and get its $S$-transform:
		\[ M^{\langle -1 \rangle}(z) = \frac{z}{\lambda+z} \text{ and thus } \mathscr{S}_p(z) = \frac{1+z}{z} M^{\langle -1 \rangle}(z)  = \frac{ 1+z}{\lambda+z}.\]

		To prove  \ref{Spi2}, we use that $\tau$ is tracial so the $S$-transform of $p apa^* p$ equals the $S$-transform of  $p^2 apa^* $ and hence $papa^*$. Then, \editr{because $a$ is $R$-diagonal,} $p$ and $a$ are free, and \editr{ hence $p$ is free from $apa^*$,} the S-transform of $p a p a^*$ factorizes as 
		\[ \mathscr{S}_{p a p a^*p}(z) = \editr{\mathscr{S}_{p}(z)\mathscr{S}_{apa^*}(z)}=  \mathscr{S}_{ p}(z)^2 \mathscr{S}_{ a  a^*}(z) .\]
		The desired result follows by applying \ref{Spi1}.
		
		To prove  \ref{Spi3}, we use that the moments of $\pi(x)$ equal the corresponding moments of $pxp$, rescaled by $ \lambda$, so:
		\[ \lambda M_{\pi(x)}(z) =  M_{pxp}(z). \]
		Then the $S$-transform is  
		\[ \mathscr{S}_{\pi(x)}(z) =   \frac{z+1}{z}  M^{\langle -1 \rangle}_{\pi(x)}(z) = \frac{z+1}{z}  M_{pxp}^{\langle -1 \rangle}(\lambda  z) = \frac{\lambda(z+1)}{\lambda z+1} \mathscr{S}_{pxp}(\lambda z) ,\]
		as desired.
	\end{proof}
	
	We now apply the above lemma to compute the $S$-transform of \\
	$\pi(a)\pi(a)^*= \pi(a)\pi(a^*)= \pi(p a p a^*p) = \pi(a p a^*)$.

	\begin{proof}[Proof of Proposition \ref{prop:brownmeasuresum}]
		
		We begin by applying \ref{Spi3} from Lemma \ref{lem:Spi} to $p a p a^*p$ and then 
		\ref{Spi2} to the result:
		
		\[ \mathscr{S}_{\pi( p a p a^*p)}(z) = \frac{\lambda(z+1)}{\lambda z+1}  \mathscr{S}_{p a p a^* p}(\lambda z)  =\frac{\lambda(z+1)}{\lambda z+1}   \frac{ (1+\lambda z)^2}{(\lambda+\lambda z)^2} \mathscr{S}_{a a^*}(\lambda z) = \frac{ 1+\lambda z}{\lambda(1+ z)} \mathscr{S}_{a a^*}(\lambda z) . \]

		Then using that $\mathscr{S}_{\lambda^{-2} \pi(a) \pi(a)^*} =  \lambda^{2}\mathscr{S}_{ \pi(a) \pi(a)^*}  $ gives
		\begin{equation} \label{eq:Spiaa} \mathscr{S}_{\lambda^{-2}\pi( a)  \pi (a)^*}(z) = \frac{ \lambda (1+\lambda z)}{1+ z} \mathscr{S}_{a a^*}(\lambda z).  \end{equation}
		Setting $\lambda = 1/k$ shows that the Brown measure of \editr{ $k \pi( a) $ and hence, by Lemma \ref{lem:brownequiv}, $a_1 + \cdots + a_k$ is $\mu_a^{\rdplus k}$, as \eqref{eq:Spiaa} is the S-transform given in \eqref{eq:Slambda}.}

		We see that $\mu_a^{\rdplus j}$ forms a semigroup by using \eqref{eq:Slambda} to compute the $S$-transforms of each side of \eqref{eq:semigroup}:
		\[  \frac{1+ z/l}{l (1+z) } \frac{1+ (z/l)/j}{j (1+z/l) } \mathscr{S}_{a a^*}\left(\frac{z/l}{j}\right) = 
		\frac{1+ z/(lj)}{lj (1+z) } \mathscr{S}_{a a^*}\left(\frac{z}{lj}\right). \]
	\end{proof}

	\subsection{Connection between Brown measures and derivatives of random polynomials} \label{sec:connection}
	
	We now relate the fractional free convolution to the roots of the derivatives of random polynomials.
	
	Given that the measures in $\mathcal{RP}_p$ arise as the limiting root distribution for polynomials with random coefficients, it makes sense to define $\mathcal{RP}_p$ as the domain of the differentiation flow. Additionally, we let $\Phi_t(r) := \frac{1}{1-t}\Psi_t(r) $, be a rescaling of $\Psi_t$ in \eqref{eq:derivativeflow}, in order to keep the total mass constant. It is easy to see that $\Phi_t$ satisfies the following definition. Throughout we will use $\Phi^{\langle -1 \rangle}$ to denote the quantile function of a radial CDF $\Phi$, see Appendix \ref{sec:quantile}.

	\begin{definition}
		Let $\mu \in \mathcal{RP}_p$ with radial CDF $\Phi_{0}$. The \textbf{differentiation flow} starting from $\mu$ is the subset $\{\mu^{t}\}_{0\leq t <1}$ of $\mathcal{RP}_p$ such that $\mu^{t}$ is the probability measure with radial CDF $\Phi_t$, and quantile functions satisfying \begin{equation}\label{eq:A:CDF identity}
			\Phi_t^{\langle-1\rangle}(x)=\frac{x(1-t)\Phi_0^{\langle-1\rangle}((1-t)x+t)}{x(1-t)+t},
		\end{equation} for $x\in(0,1)$, where $\Phi_t^{\langle -1 \rangle}$ is the quantile function of $\Phi_t$ and the existence of such measures follows from Lemma \ref{lemma:quantile facts} and the fact that the functions defined by \eqref{eq:A:CDF identity} are left-continuous and non-decreasing.  
	\end{definition} 
	
	\begin{remark}\label{rmk:A:connecting diff flow to feng-yao thm}
		Equation \eqref{eq:A:CDF identity} is a rescaled version of \eqref{eq:derivativeflow} to ensure the total mass of the associated measure is $1$. Hence, if $\mu$ is the measure arising in  part \ref{thm:part 1:feng-yao} of Theorem \ref{thm:feng-yao}, then for any $t\in(0,1)$ $\mu^{t}$ is exactly the limiting measure in part \ref{thm:part 2:feng-yao} of Theorem \ref{thm:feng-yao}.
	\end{remark}
	
	We now connect the differentiation flow to the fractional free convolution of Brown measures, which can be seen in Figure \ref{fig:A:Diagram of relationship}.  For a rotationally invariant measure $\mu$ in the complex plane, recall that $\sq\mu(\mathbb{D}_r):=\mu(\mathbb{D}_{\sqrt{r}})$ for $r > 0$, where $\mathbb{D}_r:=\{z \in \mathbb{C}: |z|< r\}$, and $\sq^{-1} \mu (\mathbb{D}_r) = \mu(\mathbb{D}_{r^2})$ is the inverse map.

	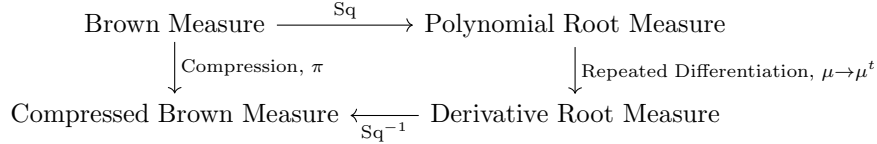
\begin{figure}[hhh] 
		\begin{tikzcd}
			\text{Brown Measure} \arrow[r, "\sq"] \arrow[d, "\text{Compression, } \pi"]
			& \text{Polynomial Root Measure} \arrow[d, "\text{Repeated Differentiation, } \mu \to \mu^{t}"] \\
			\text{Compressed Brown Measure} 
			& \text{Derivative Root Measure} \arrow[l, "\sq^{-1}"]
		\end{tikzcd} \caption{This diagram represents the relationship between the free compression of $R$-diagonal elements and repeated differentiation of random polynomials. The map $\sq$ on radially symmetric measures is defined before Theorem \ref{thm:A:Kacs free convolution}. Note, when comparing repeated differentiation directly to free compressions there is no need to include any rescaling of the roots, unlike when comparing to the convolution $\oplus$.}%
		\label{fig:A:Diagram of relationship}%
	\end{figure}
	
	\edits{For a degree $n$ random polynomial $P_n$, we will be interested in the roots of the $\lceil tn \rceil$-th derivative of $P_n$, where $t \in (0, 1)$.  It is often convenient to view $t$ as time, with $t = 0$ corresponding to the initial distribution of roots. }
	
	\begin{theorem}\label{thm:A:connection general case}
		Let \begin{equation*}
			P_n(z) = \sum_{k=0}^n \xi_k P_{k,n} z^k
		\end{equation*} be a random polynomial, with $P_{k,n}$ satisfying Assumption \ref{assump:a1} and $\xi_k$ being iid random variables satisfying  \eqref{eq:xik}, such that $\mu$ is the limiting empirical root distribution of $P_n$. Additionally, assume there exists an $R$-diagonal element $a$ affiliated to some non-commutative probability space $(\mathcal{M},\tau)$ with Brown measure $\sq^{-1}\mu$.	For any fixed $t \in (0,1)$, let $\mu_t$ be the limiting empirical root distribution of the $\lceil t n \rceil$-th derivative of $P_n((1-t)^2x)$  as $n \to \infty$ (whose existence is guaranteed by Theorem \ref{thm:feng-yao}), then \editr{\[ \mu_{t}=\sq[(\sq^{-1}\mu)^{\oplus 1/(1-t)}]. \] }
	\end{theorem}
	
	\begin{proof}
		Let $a$ be an $R$-diagonal element with Brown measure $\sq^{-1}\mu$.
		We then let $F_a(r)$ be the radial CDF of the Brown measure of $a$. From \eqref{eq:Brownmeasure}, we have that $F_a(r) =   1 + \mathscr{S}_{a^* a}^{\langle -1 \rangle}(r^{-2})$, for $ r \in [\lambda_1, \lambda_2]$. Solving for $F_a^{\langle -1 \rangle}$ gives:
		\[ F_a^{\langle -1 \rangle}(x) = \frac{1}{ \sqrt{\mathscr{S}_{a^* a}(x-1)}},\]
		for $x \in (0,1)$. 
		
		\editr{Let $\lambda =1-t$ and} $ F_\lambda $ be the radial CDF of the measure $(\sq^{-1}\mu)^{\oplus 1/\lambda}$, recalling the definition of $\mathscr{S}_k$ in \eqref{eq:Slambda} and setting $k = \lambda^{-1}$, we have from \eqref{eq:kBrownmeasure} that 
		\[ F_\lambda^{\langle -1 \rangle}(x) = \frac{1}{ \sqrt{\mathscr{S}_{\lambda^{-1}}(x-1)}},\]
		for $x \in (0,1)$.

		Evaluating \eqref{eq:Slambda} at $z = x-1$ and $k =\lambda^{-1} $ gives:
		\[  \mathscr{S}_{\lambda^{-1}}(x-1) = \frac{\lambda}{x}(\lambda (x-1)+1) \mathscr{S}_{a^* a}(\lambda(x-1)). \]
		Which in terms of $F_\lambda^{\langle -1 \rangle}$ and $F_a^{\langle -1 \rangle}(x)$  is \begin{equation}\label{eq:F evolution}
			F_\lambda^{\langle -1 \rangle}(x)=\sqrt{\frac{x}{\lambda(1+\lambda(x-1))}} F_a^{\langle -1 \rangle}((x-1)\lambda+1).
		\end{equation} 
		We let $G_\lambda$ be the radial CDF of $\sq(\sq^{-1}\mu)^{\oplus 1/\lambda}$ and $G_a$ the radial CDF of $\mu$.  We see from \eqref{eq:F evolution} that \begin{equation}\label{eq:g evolution}
			G_\lambda^{\langle -1 \rangle}(x)=\frac{x}{\lambda(1+\lambda(x-1))}G_a^{\langle -1 \rangle}((x-1)\lambda+1).
		\end{equation} After comparing \eqref{eq:g evolution} to \eqref{eq:A:CDF identity} with $t = 1- \lambda$ and initial condition $\Phi_0 = G_a$ we see that \[ G_\lambda^{\langle -1 \rangle}(x)=\frac{1}{\lambda^2}\Phi_{1-\lambda}^{\langle -1 \rangle}(x)\] for all $x\in(0,1)$. As discussed in Remark \ref{rmk:A:connecting diff flow to feng-yao thm}, $\Phi_{1-\lambda}$ is the limiting radial CDF of the $\lceil (1-\lambda)n \rceil$-th derivative of $P_n(x)$. Hence, $G_\lambda$ is the limiting radial CDF of the $\lceil (1-\lambda)n \rceil$-th derivative of $P_n(\lambda^2x)$.
	\end{proof}

	We conclude this section by translating properties of the fractional free convolution of Brown measures to the differentiation flow.

	\begin{proposition}
		Let $\mu\in\mathcal{RP}_p(\C)$. Then \begin{enumerate}[(1)]
			\item For any $t\in(0,1)$, $0$ is in the support of $\mu^{t}$.
			\item For any $t\in(0,1)$, $\mu^{t}$ has density on $\C$. In particular $\mu^{t}(\{z:|z|=r \})=0$ for any $r>0$. 
		\end{enumerate}
	\end{proposition}
	
	\begin{proof}
		This follows in a completely analogous way to the proof of Proposition \ref{prop:oplus prop} with \eqref{eq:A:CDF identity} replacing  \eqref{eq:Slambda}.			
	\end{proof}

	\section{Consequences of the theory and examples} \label{sec:examples}
	
	In this section, \editr{we consider some examples and consequences of Theorem \ref{thm:A:connection general case}. Section \ref{subsec:examples} considers specific examples with explicit distributions, while Sections \ref{sec:products and polynomials} and \ref{sec:commutator} consider more general comparisons between random polynomials and algebraic operators of $R$-diagonal elements. Finally, as Theorem \ref{thm:A:connection general case} relates differentiation to addition of $R$-diagonal elements Sections \ref{sec:limit_roots}, \ref{sec:stable}, and \ref{sec:limit of elliptic polynomials} consider limit theorems for repeated differentiation as the number of derivatives approaches the degree and stable laws.} 
	
	\subsection{Examples}\label{subsec:examples} In this section we consider specific examples to compare the sum of $R$-diagonal elements to repeated differentiation of random polynomials in Theorem \ref{thm:A:connection general case}.
	
	\subsubsection{Circular law and random Taylor polynomials}\label{sec:CLandTP}
	
	\editr{ We call an element $c \in \mathscr{M}$ a {\bf standard circular element} if its $R$-transform \editr{(recall \eqref{eq:multRdef})} is 
		\[ R_{c,c^*}(z_1,z_2) = z_1 z_2 + z_2 z_1, \]
		We call \begin{equation} \label{eq:taylor}
			P_n(z)=\sum_{k=0}^{n}\frac{\xi_kn^k}{k!}z^k,
		\end{equation}
		where $\xi_0,\xi_1,\dots$ are iid random variables satisfying \eqref{eq:xik}, the {\bf random Taylor polynomial}\footnote{The name ``random Taylor polynomial'' for this model comes from \cite{MR4242313}.}.
		In this section we will show that 
		\[ \sq( \mu_{ \pi_\lambda(c) } ) = \mu^{t} \]
		when $\lambda = 1-t$.
	}

	The Brown measure of a standard circular element is the uniform measure on the unit disk.
	A standard computation (see, for example, \cite{MR1784419} Example 5.1) shows that the $S$-transform of $cc^*$ is:
	\[ \mathscr{S}_{c c^*}(z) = \frac{1}{1+z}. \]
	From \eqref{eq:Spiaa}, we \editr{ have} that 
	\[ \mathscr{S}_{\lambda^{-2} \pi_\lambda(c) \pi_\lambda(c)^*}(z) =  \frac{\lambda(1 + \lambda z)}{1+z} \frac{1}{1+ \lambda z} =  \frac{\lambda }{1+z}, \]
	which recovers the result from \cite{MR1784419} that the free compression of a standard circular element by $\pi_\lambda$ is just $\lambda^{1/2}$ times a standard circular element. Furthermore, setting $\lambda = \editr{k^{-1}}$ verifies the well known fact that if $c_1, \ldots, c_k$ are $\ast$-freely independent standard circular elements, then $k^{-1/2} \sum_{i=1}^k c_i$ is also a standard circular element. \editr{To match this Brown measure with the roots of a random polynomial, we remove the $\lambda^{-2}$ factor to get:
		\[ \mathscr{S}_{\pi_\lambda(c) \pi_\lambda(c)^*}(z) =   \frac{1 }{\lambda (1+z)}, \]
	}
	Which from Theorem \ref{thm:BM} gives that the radial CDF of $\mu_{\pi(c)}$ is 
	\[ F_{\pi_\lambda(c)}(r) = \frac{ r^2 }{\lambda} \]
	for $r \in [0,\sqrt{\lambda}]  $.
	
	%
	The limiting root distribution the polynomials \eqref{eq:taylor} has density $\frac{1}{2\pi|z|}$ and radial CDF $\Phi(r)=r$ for $r\in[0,1]$. It is then a simple calculation to see \editr{ from \eqref{eq:A:CDF identity} that} $\mu^{t}$ has radial CDF $\Phi_{t}(r)=\frac{r}{1-t}$ for $r\in[0,1-t]$, and is just a rescaling of $\Phi$, \editr{ by $(1-t)^{-1}$. Choosing $t = 1-\lambda$, we have that $F_{\pi_\lambda(c)}(r) = \Phi_{1-\lambda}(r^2)$, as expected from Theorem \ref{thm:A:connection general case}.}

	\subsubsection{Haar unitaries} \label{sec:Haar} We now discuss the example in Theorem \ref{thm:A:Kacs free convolution} in more detail. \editr{We begin by giving examples from \cite{MR1784419} of sums of freely independent Haar unitary elements and of the product of a free projection and a Haar unitary, whose Brown were computed by less direct methods than using \eqref{eq:Slambda}. We then discuss the Kac random polynomial.}
	
	In \cite{MR1784419}, Haagerup and Larsen consider
	\[ u^{(k)} := u_1 + u_2 + \cdots+ u_k  \]
	where $u_1, u_2, \ldots, u_k$ are $\ast$-freely independent, Haar unitary elements and show that the $S$-transform of $u^{(k)}u^{(k)*}$ is 
	\begin{equation} \label{eq:Ssumhaar} S_{u^{(k)}u^{(k)*}}(z) = \frac{ z + k }{ k^2(z+1)  },   \end{equation}
	and hence the Brown measure has radial CDF
	\begin{equation} \label{eq:sumHarr} F_{\mu_{u^{(k)}}}(r) =  (k - 1)\frac{r^2}{ k^2 - r^2 } \end{equation}
	for $r \in [0,\sqrt{k}]$.
	
	On the other hand \eqref{eq:Ssumhaar} is exactly the multiplicative factor in \eqref{eq:Slambda}, so \eqref{eq:Ssumhaar} is also the \editr{$S$-transform of $u p (u p)^*$}, where $u$ is a Haar unitary element and $p$ is a $\ast$-freely independent projection with trace $\tau(p)= \editr{k^{-1}}$. 
	
	\editr{ The Brown measure of $up$ was also considered in \cite{MR1784419}. After removing the atom of the Brown measure at $0$ it follows from their result that $\pi_{\lambda}(u)$ has radial CDF \begin{equation}\label{eq:A:compressed unitary CDF}
			F_{\pi(u)}(r)=\begin{cases}
				\frac{1-\lambda}{\lambda}\frac{r^2}{1-r^2},&\ 0\leq r\leq \sqrt{\lambda}\\
				1,&\ r\geq \sqrt{\lambda}
			\end{cases}.
		\end{equation}
		Setting $k = \lambda^{-1}$, this expression, as expected, is a dilation of \eqref{eq:sumHarr} by k.}


The Brown measure of $u$ is the uniform probability measure on the unit circle in $\C$. Hence, even after applying $\sq$, the natural random polynomial to compare to is the Kac polynomial \begin{equation}
	P_n(z)=\sum_{k=1}^n \xi_kz^k,
\end{equation} where $\xi_0,\xi_1,\dots,$ are iid random variables satisfying \eqref{eq:xik}. The empirical root measure of Kac polynomials converges, see for example Theorem \ref{thm:KZ}, in probability to the uniform probability measure on the unit circle.

Let $t=1-\lambda\in(0,1)$. Feng and Yao \cite{MR3921311} established that the empirical root measure of $P_n^{\lfloor(tn)\rfloor}$ converges (see Theorem \ref{thm:feng-yao}) in probability to the measure with radial CDF\begin{equation}\label{eq:A:Kac der CDF}
	\Phi_t(r)=\begin{cases}
		\frac{t}{1-t}\frac{r}{1-r},&\ 0\leq r< 1-t\\
		1,&\ r\geq 1-t
	\end{cases}.
\end{equation} Then $\Phi_t$ is the push-forward of $F_{\pi(u)}$ under $\sq$, as given in Figure \ref{fig:A:Diagram of relationship}.

\subsection{Products of $R$-diagonal elements and random polynomials}\label{sec:products and polynomials}

\editr{In this section, we give an operation on random polynomials, namely coefficient-wise multiplication, which also describes products of free $R$-diagonal elements. Products of free $R$-diagonal elements are simpler than sums, so we are also able to consider elements which are free but not necessarily identically distributed.}

\begin{proposition}\label{prop:R-diag product as polynomials}
	Let \begin{equation*}
		P_n(z) := \sum_{k=0}^n \xi_k P_{k,n} z^k,
	\end{equation*} and \begin{equation*}
		Q_n(z) := \sum_{k=0}^n \xi_k Q_{k,n} z^k,
	\end{equation*} be random polynomials where $P_{k,n}$ and $Q_{k,n}$ are deterministic coefficients satisfying Assumption \ref{assump:a1} for some functions $P(z)$ and $Q(z)$ respectively, and $\xi_0, \xi_1, \ldots$ are iid non-degenerate complex-valued random variables which satisfy $\E \log(1 + |\xi_0|) < \infty$. Let $\mu_{P}$ and $\mu_{Q}$ be the limits \editr{(as given by Theorem \ref{thm:KZ})} of the empirical root measures of $P_n$ and $Q_n$ respectively. Define the random polynomial \begin{equation*}
		S_n(z):=\sum_{k=0}^n P_{k,n}Q_{k,n}\xi_kz^k.
	\end{equation*} Then $P_{k,n}Q_{k,n}$ satisfy Assumption \ref{assump:a1} with function $S(z):=P(z)Q(z)$, and the radial quantile function, $\Phi_{S}^{\langle-1\rangle}$, of the \editr{(again as given by Theorem \ref{thm:KZ})} limiting empirical root measure, $\mu_S$, is given by \begin{equation}
		\Phi_{S}^{\langle-1\rangle}(x)=\Phi_{P}^{\langle-1\rangle}(x)\Phi_{Q}^{\langle-1\rangle}(x),
	\end{equation} where $\Phi_{P}^{\langle-1\rangle}$ and $\Phi_{Q}^{\langle-1\rangle}$ are the radial quantile functions of $\mu_{P}$ and $\mu_{Q}$, respectively. Moreover, if $x$ and $y$ are $\ast$-free $R$-diagonal elements such that $\mu_x=\sq^{-1}\mu_P$ and $\mu_y=\sq^{-1}\mu_Q$, then $\mu_{xy}=\sq^{-1}\mu_S$.
\end{proposition}

\begin{proof}
	It is immediate to see $S$ satisfies points \ref{assump:a1 p is non neg} and \ref{assump:a1 p is continuous} of Assumption \ref{assump:a1}. For \ref{assump:a1 coefficients converge to p}, note \begin{align*}
		\lim_{n \to \infty} \sup_{0 \leq k \leq n} &\left| |P_{k,n}Q_{k,n}|^{1/n} - P\left( \frac{k}{n}\right)Q\left(\frac{k}{n}\right) \right| \\
		&\leq\lim_{n \to \infty} \sup_{0 \leq k \leq n} \Bigg[ \left| |P_{k,n}Q_{k,n}|^{1/n} - P\left( \frac{k}{n}\right)|Q_{k,n}|^{1/n} \right| \\
		&\qquad+\left| P\left(\frac{k}{n}\right)|Q_{k,n}|^{1/n} - P\left( \frac{k}{n}\right)Q\left(\frac{k}{n}\right) \right| \Bigg] \\
		&=0, 
	\end{align*} where the last equality follows from the continuity, and hence boundedness, of $P$ and $Q$. It remains to show the quantile function $\Phi^{\langle-1\rangle}_{S}$  factors as the radial quantile functions $\Phi^{\langle-1\rangle}_{P}$ and $\Phi^{\langle-1\rangle}_{Q}$. Let $I: \mathbb{R} \to \mathbb{R} \cup \{+\infty\}$ be the Legendre--Fenchel transform of $u(x) = -\log S(x)$, then \begin{equation}\label{eq:Phi_s def}
		\Phi_S(r)=I'(\log r), \quad r>0.
	\end{equation} We will assume $u$, $-\log P$, and $-\log Q$ are convex functions. Otherwise, we can \editr{ take the Legendre--Fenchel transformation twice to get new functions $\widetilde{u}$, $\widetilde{-\log P}$, and $\widetilde{-\log Q}$  which are convex.} As the Legendre--Fenchel transform is an involution on convex functions, this change has no affect on $\Phi_{S}$, $\Phi_{P}$, or $\Phi_{Q}$. To consider quantile functions, we note from general properties of Legendre--Fenchel transforms that $[I']^{\langle-1\rangle}=u'$ (see for example \cite{MR0274683}, specifically Corollary 23.5.1). Taking (generalized) inverses in \eqref{eq:Phi_s def}\begin{align*}
		\Phi^{\langle-1\rangle}_{S}(x)&=\exp\left([I']^{-1}(x)\right)=\exp\left(u'(x)\right)=\exp\left(-\frac{d}{dx}\log P(x)-\frac{d}{dx}\log Q(x)\right)\\
		&=\Phi_{P}^{\langle-1\rangle}(x)\Phi_{Q}^{\langle-1\rangle}(x).
	\end{align*} 
	\editr{To compute the Brown measure of $xy$ we use the relationship $\mathscr{S}_{xyy^*x^*} = \mathscr{S}_{xx^*}\mathscr{S}_{yy^*}$ for any $\ast$-free R-diagonal elements, \editr{this relationship follows by taking the $S$-transform of Proposition 3.6(ii) in \cite{MR1784419}}.}
	Hence, from \eqref{eq:Brownmeasure} the radial quantile function of $\mu_{xy}$ is $F^{\langle-1\rangle}_{xy}=F^{\langle-1\rangle}_{x}F^{\langle-1\rangle}_{y}$. This completes the proof of the final statement by noting $\Phi_{P}^{\langle-1\rangle}=[F^{\langle-1\rangle}_{x}]^2$ and $\Phi_{Q}^{\langle-1\rangle}=[F^{\langle-1\rangle}_{y}]^2$. 
\end{proof}
\subsection{Commutator of $R$-diagonal elements}\label{sec:commutator}

In this section, \editr{ we combine the last section with Proposition \ref{prop:brownmeasuresum} and }consider the (anti-)commutator of two free $R$-diagonal elements $x, y$. At the end of the section, we specialize to the case that $x$ and $y$ are both circular elements. 

\begin{proposition}\label{prop:commutator}
	Let $x$ and $y$ be free $R$-diagonal elements. Let $\lambda_2^x,\lambda_2^y$ be as in  Theorem \ref{thm:BM}, for $x$ and $y$, respectively, and let $ \lambda_2 = \sqrt{2} \lambda_2^x \lambda_2^y $ and 
	\[ \mathscr{S}_{comm}(z) = \frac{2 + z}{ 4(1+z)} \mathscr{S}_{x^* x}(z/2) \mathscr{S}_{y^* y} (z/2) .\] Then the Brown measure of $xy \pm yx$ is given by:
	\[ \mu_{xy \pm yx}( \D_r)   = \begin{cases} 1 + \mathscr{S}_{comm}^{\langle -1 \rangle}(r^{-2})  & \text{ if }  0 < r < \lambda_2     \\ 1 & \text{ if } r \geq \lambda_2   \end{cases}.   \]
\end{proposition}

\begin{proof}
	We prove the statement for the commutator, the anti-commutator is completely analogous. Our main goal in this proof, is to rewrite $xy -yx$ as the sum of two $\ast$-free, identically distributed \editr{$R$-diagonal} elements, \editr{we will do this through a series of reductions}.
	We begin with a standard trick when working with $R$-diagonal elements and introduce a new Haar unitary $u$, that is $\ast$-freely independent from $x$ and $y$. Then, because $x$ and $y$ are $R$-diagonal, we have that $(ux, yu^*)$ has the same $\ast$-distribution as $(x,y)$, and thus we can instead consider the Brown measure of 
	\[uxyu^* - yu^* u x = uxyu^* - yx.\]
	Furthermore, $uxyu^*$ and $y x$ are $\ast$-freely independent (see for example, Exercise 5.24 \cite{MR2266879}), so we can introduce two more $R$-diagonal elements \editr{$a,b$ such that $(a,b)$} has the same $\ast$-distribution as $(x,y)$\editr{, and then compute the Brown measure of $ab - yx$ } Then, because $yx$ and $-yx$ have the same $\ast$-distributions, we can consider the Brown measure of $ab + yx$. \editr{ The Brown measure of $yx$, and thus $ab$, as $ab$ and $yx$ have the same Brown measure}, was given in Proposition \ref{prop:R-diag product as polynomials}.
	Then since the $\ast$-distribution of $yx$ and $ab$ are the same, we can apply Proposition \ref{prop:brownmeasuresum} with $k=2$ to complete the proof.
\end{proof} 

Before specializing to commutators of circular elements, we give a polynomial interpretation of the \editr{commutator} of general $R$-diagonal elements. \editr{Let $P_n$, $Q_n$, $x$, and $y$ be exactly as in Proposition \ref{prop:R-diag product as polynomials}. Define the random polynomial \begin{equation*}
		C_n(z)=\sum_{j=0}^{\lceil n/2 \rceil}\xi_{j+\lfloor n/2 \rfloor}P_{j+\lfloor n/2 \rfloor,n}Q_{j+\lfloor n/2 \rfloor,n}\frac{(j+\lfloor n/2 \rfloor)!}{\lfloor n/2 \rfloor!4^j}z^j.
	\end{equation*} One can check using Theorem \ref{thm:A:connection general case}, Proposition \ref{prop:R-diag product as polynomials}, and the proof of Proposition \ref{prop:commutator} that $\mu_{xy\pm yx}=\sq^{-1}\mu_{C}$. }

We now consider the Brown measure of the commutator of two $\ast$-free circular elements. We note that this model was considered in \cite{MR4492979}, where it is shown that the empirical spectral distribution of any quadratic polynomial in independent Ginibre random matrices converges to the Brown measure of the corresponding polynomial in $\ast$-free circular elements, but the Brown measure was not computed. 

Since $\mathscr{S}_{xx^*}(z) = \mathscr{S}_{yy^*}(z) = \frac{1}{1+ z}$, we have that 
\[ \mathcal{S}_{comm}(z) = \frac{2 + z}{ 4(1+z)} \frac{ 4  }{(2 + z)^2 } = \frac{ 1 }{(1+z)(2+ z)}. \]
Proposition \ref{prop:commutator} then gives that the radial CDF of the Brown measure of $xy - yx$ is 
\[ \mu_{xy -yx}(\D_r) = 1  + \frac{-3 + \sqrt{1 + 4 r^2}}{ 2}  =   \frac{-1 + \sqrt{1 + 4 r^2}}{ 2}   \]
for $r \in (0,\sqrt{2})$.

\subsection{\editr{The limit of the differentiation flow} }
\label{sec:limit_roots}
In this section we consider the limiting behavior of polynomial roots as the proportion of derivatives to the degree approaches one. With Theorem \ref{thm:A:connection general case} connecting repeated differentiation of random polynomials to sums of free random variables, it is natural to consider distributions which are stable under the differentiation flow defined by \eqref{eq:A:CDF identity} and serve as central limits for the convolution.  For $\alpha \in (0,2]$, let $\mu_\alpha\in \mathcal{RP}_p(\C)$ be the measure with radial CDF $\Phi_{0,\alpha}$ such that $\Phi_{0,\alpha}^{\langle-1\rangle}(x)=\frac{x}{(1-x)^{\frac{2}{\alpha}-1}}$. It is then easy to check (recall \eqref{eq:A:CDF identity}) that \begin{equation*}
	\Phi_{t,\alpha}^{\langle -1\rangle}(x)=(1-t)^{2-\frac{2}{\alpha}}\Phi_{0,\alpha}^{\langle-1\rangle}(x),
\end{equation*} for all $x,t\in(0,1)$. Hence, $\mu_\alpha^{t}$ is $\mu_\alpha$, up to a $t$-dependent rescaling of the support, and we refer to $\mu_\alpha$ as \textbf{stable} under the differentiation flow. \editr{Of course, scalar dilations of stable laws are also stable. So we provide the following general definition of differentiation stable. \begin{definition}
		Let $\mu$ be a rotationally invariant probability measure with invertible radial CDF $\Phi$. $\mu$ is said to be $\alpha$-differentiation stable for $\alpha\in (0,2]$ if there exists some $\theta>0$ such that\begin{equation}
			\Phi^{\langle-1\rangle}(x)=\theta\frac{x}{(1-x)^{\frac{2}{\alpha}-1}},
		\end{equation} for all $x\in[0,1)$. 
	\end{definition}
	As the following proposition demonstrates this definition of $\alpha$-differentiation stable is consistent with the already existing notion of $\alpha$-$\oplus$ stable and Theorem \ref{thm:A:connection general case}. \begin{proposition}
		A radially symmetric probability measure $\mu$ is $\alpha$-differentiation stable if and only if $\sq^{-1}\mu$ is $\alpha$-$\oplus$ stable.
	\end{proposition}
	\begin{proof}
		The proof is an immediate consequence of Theorem \ref{thm:BM} and Proposition \ref{prop:stablelaw}.
\end{proof}}

For simplicity of presentation we consider initial root distributions with power law tail decay and compact support/super-polynomial tail decay separately in Theorem \ref{conjecture:general clt} and Corollary \ref{thm:A:Derivative CLT} respectively. The proofs however are nearly identical, with the power law decay requiring only a short extra initial discussion on regularly varying functions\footnote{See, for example, \cite{MR2364939} for background on regularly varying and slowly varying functions.}.

\begin{theorem}[Limit of repeated differentiation]\label{conjecture:general clt}
	Let $\mu \in \mathcal{RP}_p(\C)$ with unbounded support and radial CDF $\Phi_0$ such that\begin{equation}\label{eq:A:tail assumption}
		\lim\limits_{r\rightarrow \infty}\frac{1-\Phi_0(r)}{L(r)r^{-\frac{\alpha}{2-\alpha}}} =1,
	\end{equation} for some $\alpha\in (0,2)$ and some positive slowly varying function $L$. Let $\{\Phi_t\}_{t\in[0,1)}$ be the family defined by \eqref{eq:A:CDF identity}. Then, there exists a positive slowly varying function $g:[1,\infty)\rightarrow[1,\infty)$ such that the radially symmetric probability measures $\widetilde{\mu}^t$ with radial CDF $\widetilde{\Phi}_t(x)=\Phi_t\left(g((1-t)^{-1})(1-t)^{2-\frac{2}{\alpha}}x \right)$ converges weakly to the probability measure $\mu_{\alpha}$ with radial quantile function  $\Phi_{0,\alpha}^{\langle-1\rangle}(x)=\frac{x}{(1-x)^{\frac{2}{\alpha}-1}}$ as $t \rightarrow 1^{-}$.
\end{theorem} 
To understand the rescaling in Theorem \ref{conjecture:general clt} it is helpful to rewrite $(1-t)^{2-\frac{2}{\alpha}}$ as $(1-t)(1-t)^{-(\frac{2}{\alpha}-1)}$. The $(1-t)^{-(\frac{2}{\alpha}-1)}$ term is to manage the tail decay of the measure, as described by \eqref{eq:A:tail assumption}. The $(1-t)$ term corrects for the natural flow inward of polynomial roots under differentiation, as described by the Gauss-Lucas theorem.  

Before the proof of Theorem \ref{conjecture:general clt} we give a brief discussion of how it may be interpreted as some type of central limit theorem. For simplicity assume $g$ is the constant function $g(x)=1$ First, let $x_1,x_2,\dots$ be some sequence of freely independent identically distributed $R$-diagonal elements such that $\mu_{x_1}=\sq^{-1}\mu$. If $t=1-\frac{1}{k}$ for some natural number $k$, then from Theorem \ref{thm:A:connection general case} we see that $\sq^{-1}\widetilde{\mu}^t$ is \editr{the Brown measure of \[\frac{x_1+\cdots+x_k}{k^{\frac{1}{\alpha}}}.\] } Thus, taking the $t\rightarrow 1^{-}$ limit of $\widetilde{\mu}^t$ is essentially the same as taking the $k\rightarrow\infty$ limit of $\frac{x_1+\cdots+x_k}{k^{\frac{1}{\alpha}}}$. In this way Theorem \ref{conjecture:general clt} is essentially a generalized central limit theorem for $R$-diagonal elements translated, through Theorem \ref{thm:A:connection general case}, to the language of repeated differentiation of random polynomials.

\begin{proof}[Proof of Theorem \ref{conjecture:general clt}]
	From \eqref{eq:A:tail assumption} we have that $\frac{1}{1-\Phi_0}$ is $\frac{\alpha}{2-\alpha}$-varying. Thus, see Resnick \cite{MR2364939} Proposition 0.8, the function $y\mapsto \Phi_0^{\langle-1\rangle}\left(1-\frac{1}{y} \right)$ is $\frac{2-\alpha}{\alpha}$-varying in $y$. Thus there exists a positive slowly varying function $g:[1,\infty)\rightarrow(0,\infty)$ such that $\Phi_0^{\langle-1\rangle}\left(1-\frac{1}{y} \right)\sim g(y)y^{\frac{2-\alpha}{\alpha}}$ as $y\rightarrow\infty$. Defining the function $f(t)=g\left(\frac{1}{1-t}\right)^{-1}$, it is then straightforward to check that \begin{equation}\label{eq:A:inverse tail}
		\lim\limits_{x\rightarrow1^{-}}f(x)(1-x)^{\frac{2}{\alpha}-1}\Phi_{0}^{\langle-1\rangle}=1,\quad\text{and}\quad	\lim_{t \rightarrow 1^-}\frac{f(t)}{f((1-t)x+t)}=1 
	\end{equation} for every $x\in(0,1)$.
	
	Fix $x\in(0,1)$.  We have \begin{align*}
		\lim_{t \rightarrow 1^-}\widetilde{\Phi}_{t}^{\langle -1\rangle}(x)&=\lim_{t \rightarrow 1^-}f(t)(1-t)^{-(2-\frac{2}{\alpha})}\Phi_{t}^{\langle-1\rangle}(x)\\
		&=\lim_{t \rightarrow 1^-}\frac{x}{(1-t)x+t}f(t)(1-t)^{\frac{2}{\alpha}-1}\Phi_0^{\langle-1\rangle}((1-t)x+t)\\
		&=\lim_{t \rightarrow 1^-}\frac{x}{(1-t)x+t}\cdot\frac{f(t)(1-t)^{\frac{2}{\alpha}-1}}{f((1-t)x+t)(1-((1-t)x+t))^{\frac{2}{\alpha}-1}}\\
		&=\frac{x}{(1-x)^{\frac{2}{\alpha}-1}},
	\end{align*} where the third equality follows from \eqref{eq:A:inverse tail}. Hence $\widetilde{\Phi}_t^{\langle -1\rangle}$ converges to $\Phi_{0,\alpha}^{\langle -1 \rangle}$ pointwise on $(0,1)$. It then follows from Lemma \ref{lemma:quntile convergence} that $\widetilde{\Phi}_t$ converges to $\Phi_{0,\alpha}$ pointwise on $(0,\infty)$ as $t\rightarrow 1^-$. This completes the proof. 
\end{proof} 	

The following is stated as a corollary of Theorem \ref{conjecture:general clt} with $\alpha=2$, i.e., measures with compact support or super-polynomial tail decay. In these cases \eqref{eq:A:tail assumption} is ill-defined, however looking at the proof it is still true that $y\mapsto \Phi_0^{\langle-1\rangle}\left(1-\frac{1}{y} \right)$ is $0$-varying in $y$. Hence, the proof of Corollary \ref{thm:A:Derivative CLT} is identical to the proof of Theorem \ref{conjecture:general clt} beginning from the second line. For simplicity we state it only for measures with compact support. It is worth noting the limit is $\sq$ applied to the uniform distribution on the unit disk in the complex plane, i.e., the Circular Law, one of the most important measures in non-Hermitian free probability theory. 

\begin{corollary}\label{thm:A:Derivative CLT}
	Let $\mu$ be a probability measure in $\mathcal{RP}_p(\C)$ with radial cumulative distributed function $\Phi_{0}$ such that $\inf\{x\geq0:\Phi_0(x)=1\}=1$. Then \begin{equation}\label{eq:A:limit CDF result}
		\lim_{t \rightarrow 1^-}\Phi_t((1-t)r)=r,
	\end{equation} for any fixed $r\in(0,1)$, where $\Phi_t$ is the radial CDF of $\mu^{t}$.
\end{corollary}

\subsection{Stable laws} \label{sec:stable}

In \cite{MR3896715}, the Brown measures that are stable under the $\rdplus$ operation are characterized. One way to establish this characterization is via the bijection $\mathcal{H}$, so it suffices to determine the symmetric probability measures that are stable under $\boxplus$, and then determine their $S$-transforms. 
This was done in \cite{MR2506464}, building upon the work of \cite{MR2128863}, where it is shown that the $S$-transform of any symmetric free stable distribution is of the form:
\[ \mathscr{S}_{\alpha}(z) = \theta e^{i(2-\alpha) \frac{ \pi}{2 \alpha}} z^{\frac{1}{\alpha}-1} = \theta i (-z)^{1/\alpha-1} \]
for $0 < \alpha \leq 2 $, and some constant $\theta > 0$. Additionally, they prove that for \editr{ $\widetilde \nu$, a symmetric probability measure on $\R$:
	\[\mathscr{S}_{\widetilde \nu}^2(z) = \frac{1+z}{z} \mathscr{S}_{\nu}(z) ,  \]
	where $\nu$ is the probability measure on $\R^+$ induced by the map $x \to x^2$, as given before \eqref{eq:oplus definition}.}
So we see that if $\widetilde \nu$ is a symmetric free stable distribution then the $S$-transform of $\nu$ is:
\begin{equation} \label{eq:stableS2} \mathscr{S}_{\nu}(z) = \theta \frac{ (-z)^{\frac{2}{\alpha} -1}}{1+z}\end{equation}
for some (possibly different) constant $\theta > 0$.  
This is exactly the expression in \eqref{eq:stableS}. We now give an alternative, more direct, proof of Proposition \ref{prop:stablelaw}, \editr{that the Brown measure of $a$ is $\alpha$-$\rdplus$ stable when the $S$-transform of $\mu_{aa^*}$ is of the form \eqref{eq:stableS2}}, without relying on the characterization of stable laws on $\R$.

\begin{proof}[Proof of Proposition \ref{prop:stablelaw}]
	We will show that the $S$-transforms in \editr{\eqref{eq:stableS}, which we remind the reader is just \eqref{eq:stableS2} with $\nu$ replaced by $aa^*$,} are exactly the class of functions that are \editr{invariant under applying $\pi$ and rescaling as in} \eqref{eq:Spiaa}. 
	
	If $\mathscr{S}_{a a^*} (z) = \frac{(-z)^{\frac{2}{\alpha}-1}}{1+z}$ then a simple rescaling of the argument gives:
	\[  \mathscr{S}_{a a^*} (z) = \frac{\lambda^{-2/\alpha} \lambda (1+\lambda z)}{1+z} \mathscr{S}_{a a^*} (\lambda z) .\] 
	Additionally, these are exactly the class of functions that are invariant after  rescaling by $\lambda$ and then multiplying by $c \frac{ \lambda (1+\lambda z)}{1+z} $, for some constant $c$.
	But from \eqref{eq:Spiaa}, we have that the right-hand side is equal to the $S$-transform of $\lambda^{2/\alpha-2} \pi_{\lambda}(a)\pi_{\lambda}(a)$. In other words, the Brown measure is \editr{invariant under} the map $a \to \lambda^{1/\alpha-1} \pi_{\lambda^{}}(a)$, which, by Proposition \ref{prop:brownmeasuresum}, has the same law as
	\[ \frac{a_1 + \cdots + a_k  }{ k^{1/\alpha}  } \]
	when $\lambda = \editr{k^{-1}}$.
\end{proof}
The case $\alpha =2$ corresponds to the circular element, considered in Section \ref{sec:CLandTP}. More generally, in \cite{MR3896715}, it is shown that if $l = \frac{2}{\alpha}-1$ is an integer, then the Brown measure of $x_0 x_1^{-1} \cdots x_{l}^{-1}$, where the $x_i$'s are $\ast$-freely independent circular elements\editr{, is $\alpha$-$\rdplus$ stable}.	

\editr{In the remainder of this section, we use Proposition \ref{prop:R-diag product as polynomials} to directly relate the measures $\mu_\alpha$ appearing in Theorem \ref{conjecture:general clt} to random polynomials with independent coefficients. In principle, random polynomials with independent coefficients whose empirical root measures converge to $\mu_\alpha$ could be reverse engineered from the radial quantile function and Theorem \ref{thm:KZ}. We will instead use Proposition \ref{prop:R-diag product as polynomials} and already known results on $\alpha$-$\oplus$ stable laws to construct simple coefficients for these polynomials which transform quite nicely under differentiation.}

\editr{ For $l = \frac{2}{\alpha}-1$ an integer, we let $x_0,x_1,\dots,x_l$ be free circular elements. We have already seen in Section \ref{sec:CLandTP} that the random polynomial \begin{equation*}
		C_{n}(z)=\sum_{k=0}^{n}\frac{n^k}{k!}\xi_k z^k,
	\end{equation*} have limiting root measure $\sq\mu_{x_0}$. One can verify through the $S$-transform and and radial quantile function that the random polynomial \begin{equation*}
		D_{n}(z)=\sum_{k=0}^{n}\frac{k!}{n^k}\binom{n}{k}\xi_k z^k
	\end{equation*} has limiting root measure $\sq\mu_{x_1^{-1}}$. Applying Proposition \ref{prop:R-diag product as polynomials} we see that the random polynomial \begin{equation}\label{eq:stable polynomials}
		S_n(z)=\sum_{k=0}^{n}\left(\frac{k!}{n^k} \right)^{l-1}\binom{n}{k}^{l}\xi_k z^{k},
	\end{equation} has limiting root measure $\sq\mu_{x_0 x_1^{-1} \cdots x_{l}^{-1}}$, and hence is differentiation stable. Of course \eqref{eq:stable polynomials} makes sense for any $l\geq 0$. The stability of $S_n$ can also be seen directly by a straightforward computation \begin{align*}
		S_n'(z)&=\sum_{k=1}^{n}k\left(\frac{k!}{n^k} \right)^{l-1}\binom{n}{k}^{l}\xi_k z^{k-1}=\sum_{k=1}^{n}\left(\frac{k-1!}{(n-1)^{k-1}} \right)^{l-1}n\left(1-\frac{1}{n} \right)^{(k-1)(l-1)}\binom{n-1}{k-1}^{l}\xi_k z^{k-1}\\
		&\overset{d}{=}nS_{n-1}\left(\left(1-\frac{1}{n} \right)^{l-1}z \right).
\end{align*}}		

\subsection{Limits of elliptic polynomials and some questions of Feng and Yao}\label{sec:limit of elliptic polynomials}		
\editr{Theorem \ref{conjecture:general clt} considers root behavior as the ratio of derivatives to the degree, $t$, tends to $1$ \textit{after} the degree is sent to infinity. The limiting root measure when $t$ tends to $1$ simultaneously with the degree tending to infinity may not always converge to one of the measures in Theorem \ref{conjecture:general clt}. See the discussion after Theorem 5 in \cite{MR3921311}. In this section we consider some examples of this simultaneous limit and remark on some questions posed in \cite{MR3921311}. Specifically we consider coefficients \begin{equation}
		P_{k,n}=\binom{n}{k}^{w},
	\end{equation} for $w\geq 0$. Kabluchko and Zaporozhets \cite{MR3262481} refer these as \textbf{elliptic polynomials}.}  In \cite{MR3921311} Theorem 6, Feng and Yao computed the limiting root distribution for elliptic polynomials with \editr{$w=\frac{1}{2}$}, as the proportion of derivatives tends to one with the degree. The limit, after rescaling, in their work is the measure $\mu_{1/2}$ from Theorem \ref{conjecture:general clt}. 		
Feng and Yao \cite{MR3921311} additionally consider the limiting root distribution for derivatives of random Kac polynomials (introduced in Section \eqref{sec:Haar})\editr{, which can be viewed as $w=0$ elliptic polynomials,} as the proportion of derivatives tend to one as the degree tends to infinity. The limit, \editr{again} after rescaling, matches that of Corollary \ref{thm:A:Derivative CLT}. Without rescaling, \editr{the limiting empirical root measure of both models} would be a point mass at $0$. Feng and Yao pose the following questions for a general random polynomial as defined in \eqref{eq:pngen}:\begin{enumerate}[(1)]
	\item If $N_n=n-D_n$ where $D_n=o(n)$ and $D_n\rightarrow\infty$ with $n$, then when is $\delta_0$ the limiting root distribution of $P_n^{(\lfloor N_n\rfloor)}$?
	
	\item If the limiting root distribution of $P_n^{(\lfloor N_n\rfloor)}$ is $\delta_0$, does there exist a rescaling of the roots giving a non-degenerate limit? If so, what is the scaling?
\end{enumerate} Theorem \ref{conjecture:general clt} suggests that the answer to both questions should depend largely on the tail of the limiting root measure. However, to fully answer both questions some additional regularity on the coefficients would need to be considered (\editr{again} see the discussion following Theorem 5 in \cite{MR3921311}). In the following proposition we partially answer these questions. 


\begin{proposition}\label{prop:A:limit with n for elliptic}
	For $\alpha\in(0,2]$, let $P_{n}$ be the general random polynomial defined in \eqref{eq:pngen} such that $\xi_1,\xi_2,\dots$  satisfy \eqref{eq:xik} and \begin{equation}
		P_{k,n}=\binom{n}{k}^{\frac{2}{\alpha}-1}.
	\end{equation} Let $N_n=n-D_n$ where $D_n$ is such that $D_n=o(n)$ and $D_n\rightarrow\infty$ as $n\rightarrow\infty$, and let $R_n=\left(\frac{n}{D_n} \right)^{2-\frac{2}{\alpha}}$. Let $\widetilde{P}_n$ be the random polynomial defined by $\widetilde{P}_n(z)=P_n(z/R_n)$. Then the empirical root measure of $\widetilde{P}_n^{(\lfloor N_n\rfloor)}$ converges in probability as $n \to \infty$ to the probability measure $\mu_{\alpha}$ defined in Theorem \ref{conjecture:general clt}.
	
	Moreover, if 	\[ \mu_{D_n} = \frac{1}{D_n}\sum_{z \in \mathbb{C} : P_n^{(\lfloor N_n\rfloor)}(z) = 0} \delta_{z}, \] then weakly, in probability,\begin{equation}
		\lim\limits_{n\rightarrow \infty} \mu_{D_n}=\begin{cases}
			0,\ 0<\alpha<1\\
			\mu_1,\ \alpha=1\\
			\delta_0,\ 1<\alpha\leq 2						
		\end{cases},
	\end{equation} where $\delta_0$ is a point mass at the origin, $0$ is the zero measure, and $\mu_1$ is defined in Theorem \ref{conjecture:general clt}.
\end{proposition}

Letting $t=\frac{N_n}{n}$, and noting $\frac{n}{D_n}=(1-t)^{-1}$ we see that the scaling in Proposition \ref{prop:A:limit with n for elliptic} matches that of Theorem \ref{conjecture:general clt}. We expect this phenomenon holds in more generality, i.e., under sufficient regularity conditions on the coefficients of $P_n$ the scaling should depend only on the tail of the limiting root measure and the limiting empirical measure for the rescaled roots should be $\mu_\alpha$.
\begin{proof}[Proof of Proposition \ref{prop:A:limit with n for elliptic}]
	To simplify notation, let $\beta=\frac{2}{\alpha}-1$. Proposition \ref{prop:A:limit with n for elliptic} is a straightforward generalization of Theorem 6 in \cite{MR3921311}. We sketch the necessary changes. The $\left(\frac{n}{D_n}\right)^{-\beta}$ term in the rescaling is used to control the coefficients $P_{k,n}$, $N_n\leq k\leq n$ of $P_n$, while the additional $\left(\frac{n}{D_n}\right)$ term is used to control how the coefficients evolve under differentiation. We can then conclude that if $\widetilde{P}_{k,n}$ are the coefficients of $\widetilde{P}_n$, then\begin{equation}\label{eq:A:log limit}
		\lim\limits_{n\rightarrow \infty}\sup_{0 \leq k \leq D_n}\left|\frac{1}{D_n}\log \widetilde{P}_{k,n}-\log \widetilde{P}_\beta\left(\frac{k}{D_n}\right) \right|=0, 
	\end{equation} where \begin{equation}\label{eq:A:limiting coefficient function}
		\log P_\beta(x)=-x\log x-\beta(1-x)\log(1-x)+(1-\beta)x+\beta-1
	\end{equation} for $0\leq x\leq 1$ and $\log P_\beta(x)=-\infty$ for $x>1$. The conclusion then follows from \eqref{eq:A:log limit}, \eqref{eq:A:limiting coefficient function}, Theorem \ref{thm:feng-yao} and the observation in \cite{MR3262481} that the limiting empirical root measure is given by the push-forward of the  Lebesgue measure under the map $x\mapsto \exp\left(-\frac{d}{d\editr{x}}\log P_\beta(x) \right)$.
\end{proof}

%
%
%
%

\section{PDEs describing the limiting behavior of the roots} \label{sec:PDE}
In this section, we focus on PDEs describing the dynamics of the limiting radial probability density functions and radial cumulative distribution functions under repeated differentiation.  

Given the initial radial density $\psi(x,0)$ of the zeros at $t = 0$, the PDE in \eqref{eq:OS} describes the radial density $\psi(x, t)$ at time $0 \leq t < 1$.  
As shown in \cite{MR4242313}, there is a constant loss of mass for the solution: 
\[ \frac{d}{dt} \int_0^\infty \psi(x,t) \d x = -1. \]
In other words, if $\psi(x, 0)$ is the \edits{radial part} of a probability density function (PDF), then $\psi(x, t)$ has mass $1 - t$.  One can renormalize so that $\psi(x,t)$ has total mass $1$, but this new function will not satisfy \eqref{eq:OS}.  In this section, we informally derive new PDEs for this PDF and its corresponding CDF.  We will also derive a PDE for the Brown measure.  Although the derivations are informal, the purposes of this section is to show how our results and examples from the previous sections are consistent with the PDE approach \cite{MR4011508,MR4242313} to studying repeated differentiation of random polynomials.  

\subsection{Derivation of PDE for the PDF and CDF} \label{subsec:PDFCDF}
We define the PDF as 
\[ \varphi(x,t) := \frac{ \psi(x,t) }{1 - t}, \qquad x \geq 0, \quad 0 \leq t < 1, \]
where $\psi(x,t)$ is a solution to \eqref{eq:OS}.  
Recall that we use the convention that $x \geq 0$ either denotes $x \in [0, C]$ (for some finite positive constant $C$) or $x \in [0, \infty)$, depending on whether the density is compactly supported or not. 
The function $\varphi(x,t)$ will then satisfy the equation 
\begin{equation} \label{eq:OS:PDF}
	{ (1 - t) \frac{ \partial \varphi(x,t) }{\partial t} = \frac{ \partial}{\partial x} \left( \frac{ \varphi(x,t) }{ \frac{1}{x} \int_0^x \varphi(y,t) dy } \right) + \varphi(x,t), \qquad x \geq 0, \quad 0 \leq t < 1. } 
\end{equation} 
Indeed, from \eqref{eq:OS}, we derive
\begin{align*}
	\frac{\partial \varphi(x,t)}{\partial t} &= \frac{1}{1 - t} \frac{ \partial \psi(x,t) }{\partial t} + \frac{1}{(1 - t)^{2}} \psi(x,t) \\
	&= \frac{1}{1 - t} \frac{ \partial}{\partial x} \left( \frac{ \psi(x,t) }{ \frac{1}{x} \int_0^x \psi(y,t) dy } \right) + \frac{1}{1 -t} \varphi(x,t) \\
	&= \frac{1}{1 - t} \frac{ \partial}{\partial x} \left( \frac{ \varphi(x,t) }{ \frac{1}{x} \int_0^x \varphi(y,t) dy } \right) + \frac{1}{1 -t} \varphi(x,t).  
\end{align*}
Thus, by rearranging, we obtain \eqref{eq:OS:PDF}.  

It is easy to check that if $\int_0^\infty \varphi(x,0) \d x = 1$, then the solution to \eqref{eq:OS:PDF}  satisfies $\int_0^\infty \varphi(x,t) \d x = 1$ for all $0 \leq t < 1$.  In fact, from \eqref{eq:OS:PDF}, we have
\begin{align*}
	(1 - t) \frac{ \partial}{\partial t} \int_0^\infty \varphi(x,t) \d x &= \int_0^\infty (1-t) \frac{ \partial \varphi(x,t)}{\partial t} \d x \\
	&= \int_0^\infty  \frac{ \partial}{\partial x} \left( \frac{ \varphi(x,t) }{ \frac{1}{x} \int_0^x \varphi(y,t) dy } \right) \d x + \int_0^\infty \varphi(x,t) \d x \\
	&= -\lim_{\eps \to 0} \frac{ \varphi(\eps, t)}{\frac{1}{\eps} \int_0^\eps \varphi(y,t) \d y } + \int_0^\infty \varphi(x,t) \d x \\
	&= -1 + \int_0^\infty \varphi(x,t) \d x,
\end{align*}
where we used the regularity of the solution, and we assumed $x \mapsto \varphi(x, 0)$ is compactly supported, which by the Gauss--Lucas theorem hints that the support of $x \mapsto \varphi(x, t)$ is contained in the support of $x \mapsto \varphi(x, 0)$ for all $0 \leq t < 1$.  Thus, if $y(t) = \int_0^\infty \varphi(x,t) \d x$, we obtain the linear ODE:
\[ (1 - t) y' = y -1, \]
which admits the solution
\[ y(t) = \frac{C}{1 - t} + \frac{t}{t-1} \]
for a constant $C$ depending on the initial value.  In fact, if $y(0) = 1$, then $C = 1$, and we find the constant solution $y(t) = 1$ for all $0 \leq t < 1$, as desired.  

Define the cumulative distribution function of the solution $\varphi(x,t)$ of \eqref{eq:OS:PDF} as 
\[ \Phi(x,t) = \int_0^x \varphi(y,t) dy. \]
Then, using \eqref{eq:OS:PDF}, we obtain
\begin{align*}
	(1 - t) \frac{\partial \Phi(x,t) }{\partial t} &= \int_0^x (1 - t) \frac{\partial \varphi(y,t) }{\partial t} dy \\ 
	&= \int_0^x \frac{ \partial }{\partial y} \left( \frac{ \varphi(y, t) }{ \frac{1}{y} \int_0^y \varphi(z, t) dz } \right) dy + \int_0^x \varphi(y, t) dy \\
	&= \int_0^x \frac{ \partial }{\partial y} \left( y \frac{ \frac{ \partial \Phi(y,t)}{\partial y} }{ \int_0^y \varphi(z, t) dz } \right) dy + \Phi(x,t) \\
	&= \frac{x \frac{ \partial \Phi(x,t) }{\partial x} }{\Phi(x,t) } - 1 + \Phi(x,t).
\end{align*} 
We conclude that $\Phi(x,t)$ satisfies the following PDE: 
\begin{equation} \label{eq:OS:CDF}
	{ (1 - t) \frac{\partial \Phi(x,t) }{\partial t} = \frac{x \frac{ \partial \Phi(x,t) }{\partial x} }{\Phi(x,t) } - 1 + \Phi(x,t), \qquad x \geq 0, \quad 0 \leq t < 1. } 
\end{equation}
\edits{Equation \eqref{eq:OS:CDF} is similar to the PDE derived in Section 2.3 of \cite{doi:10.1080/10586458.2021.1980752}, where the CDF is not normalized to have total mass $1$. }
Now that we have a PDE for the CDF, we can compare \eqref{eq:OS:CDF} to the examples in the previous sections.  For instance, it is easy to check that $\Phi(x,t) = \Phi_{0, 1}(x) = \frac{x}{1+x}$ from Section \ref{sec:limit_roots} satisfies \eqref{eq:OS:CDF}.  

\subsection{Rescaling the $x$ coordinate}
If $\Phi(x, 0)$ at $t = 0$ is supported on $x \in [0, 1]$, we expect that $\Phi(x, t)$ is supported on $x \in [0, 1 - t]$ for $0 \leq t < 1$.  If we define $\widetilde \Phi(x,t) = \Phi((1-t)x, t)$, then $\widetilde \Phi(x, t)$ is supported on $x \in [0, 1]$ for all $0 \leq t < 1$ but will no longer satisfy \eqref{eq:OS:CDF}.  One can easily derive the PDE that $\widetilde \Phi(x, t)$ does satisfy using \eqref{eq:OS:CDF}. Indeed, by the chain rule, we have
\begin{align*}
	\frac{ \partial \widetilde \Phi(x,t)}{\partial t} &= -x \frac{ \partial \Phi( (1-t)x,t)}{\partial x} + \frac{\partial \Phi((1-t)x,t)}{\partial t}, \\
	\frac{ \widetilde \Phi (x,t)}{\partial x} &= (1-t) \frac{\partial \Phi((1-t)x, t) }{\partial x}. 
\end{align*}
In particular, this implies that
\[ (1 -t) \frac{ \partial \widetilde \Phi(x,t)}{\partial t} = -x \frac{\partial \widetilde \Phi(x,t)}{\partial x} + (1 - t) \frac{\partial \Phi((1-t)x,t)}{\partial t}. \]
Thus, from \eqref{eq:OS:CDF}, we obtain the following PDE for $\widetilde \Phi (x,t)$:
\begin{equation} \label{eq:OS:CDF:rescale}
	{(1 - t) \frac{ \widetilde \Phi (x,t) }{\partial t} = -x \frac{\partial \widetilde \Phi(x,t)}{\partial x} + \frac{ x \frac{\partial \widetilde \Phi(x,t)}{\partial x}}{\widetilde \Phi(x,t)} - 1 + \widetilde \Phi(x,t), \qquad x \geq 0, \quad 0 \leq t < 1. }
\end{equation}
One can now take the $t \to 1^-$ limit in \eqref{eq:OS:CDF:rescale}.  Indeed, if $\widetilde \Phi(x) = \lim_{t \to 1^-} \widetilde \Phi(x,t)$ and  $\lim_{t \to 1^-} \frac{\partial \widetilde \Phi(x,t)}{\partial x} =   \widetilde \Phi'(x)$, then one arrives at the following ODE for $\widetilde \Phi(x)$:
\begin{equation} \label{eq:OS:tode}
	x \widetilde \Phi'(x)= \frac{ x \widetilde \Phi'(x)}{\widetilde \Phi(x)} - 1 + \widetilde \Phi(x), \qquad x \geq 0. 
\end{equation}
It is straightforward to check that the example $\widetilde \Phi(x) = x$ from Theorem \ref{thm:A:Derivative CLT} solves \eqref{eq:OS:tode}.

In a similar fashion, using \eqref{eq:OS:PDF}, one can also derive a PDE for the rescaled PDF $\widetilde \varphi(x, t) = (1-t) \varphi((1-t)x, t)$ and take the limit $t \to 1^-$.  

\subsection{CDF for the Brown measure}
Using the connection between random polynomials and the Brown measure discussed in Section \ref{sec:connection}, we can similarly derive PDEs for the radial parts of the PDF and CDF for the Brown measure.  
Let
\[ F(x,t) = \Phi(x^2, t), \qquad x \geq 0, \quad 0 \leq t < 1 \]
be the CDF of the radial part of the Brown measure, where $\Phi(x, t)$ is the \edits{radial part of the CDF,} defined in Section \ref{subsec:PDFCDF} above.  Using \eqref{eq:OS:CDF}, we find
\begin{align*}
	(1 - t) \frac{\partial F(x,t)}{\partial t} &= (1 - t) \frac{ \partial \Phi(x^2, t)}{\partial t} \\
	&= \frac{x^2 \frac{ \partial \Phi(x^2, t)}{\partial x} }{ \Phi(x^2, t) } - 1 + \Phi(x^2, t).
\end{align*}
Thus, since 
\[ \frac{\partial F(x,t)}{\partial x} = 2x \frac{\partial \Phi(x^2, t)}{\partial x}, \] 
we conclude with the following PDE for $F(x, t)$: 
\[ { (1 - t) \frac{\partial F(x,t)}{\partial t} = \frac{ x \frac{ \partial F(x, t) }{ \partial x}}{ 2 F(x,t) } - 1 + F(x,t), \qquad x \geq 0, \quad 0 \leq t < 1. } \]

Similarly, using \eqref{eq:OS:PDF}, one can also derive a PDE for the PDF of the radial part of the Brown measure given by $f(x,t) = 2x \varphi(x^2, t)$; we omit the details.

\section{Funding}

This work was supported by the National Science Foundation [Grant No. DMS-2143142 to S.O.]; and the European Research Council [Grant No. 101020331].			

\section*{Acknowledgements}
The third author acknowledges the support of the University of Colorado Boulder, where a portion of this work was completed.  The authors thank Martin Auer, Vadim Gorin, Brian Hall, and Noah Williams for comments, corrections, and references.  The authors also wish to thank the anonymous referees for useful feedback and corrections.

\appendix

\section{Quantile functions} \label{sec:quantile}
In this section, we review some basic facts used throughout the paper concerning quantile functions of real-valued random variables. 

\begin{definition}
	Let $F$ be the CDF of a real valued random variable. The quantile function $Q:[0,1)\rightarrow\R$ of $F$ is the function defined by \begin{equation}\label{eq:Quantile def}
		Q(p):=\inf\{x\in \R: F(x)\geq p\}. 
	\end{equation} 
\end{definition}

The following lemma contains some essential results on quantile functions. 

\begin{lemma}\label{lemma:quantile facts}
	Let $F$ be a CDF with quantile function $Q$. \begin{enumerate}[(1)]
		\item\label{eq:inequality iff} For every $x\in\R$ and $p\in[0,1)$, $F(x)\geq p$ if and only if $Q(p)\leq x$.
		
		\item $Q$ is left-continuous and non-decreasing.
		
		\item If $F$ is invertible, then $Q=F^{-1}$.
	\end{enumerate} Moreover, the quantile function uniquely determines $F$ and any left-continuous non-decreasing function on $[0,1)$ is the quantile function of a unique distribution.
\end{lemma}

\begin{proof}
	See \cite{vaart_1998} Lemma 21.1 for a proof of the first three statements. For the final statements, let $Q$ be a left-continuous non-decreasing function on $[0,1)$. Define the function $F:\R\rightarrow[0,1]$ by \begin{equation}\label{eq:quantile to cdf def}
		F(x)=\max\left(\sup\{p\in[0,1): Q(p)\leq x\},0\right),
	\end{equation} with the convention that $\sup\emptyset=-\infty$. It is straightforward to check $F$ is non-decreasing, \begin{equation*}
		\lim_{x\rightarrow-\infty} F(x)=0,\quad\text{and}\quad \lim\limits_{x\rightarrow\infty}F(x)=1.
	\end{equation*} Fix $x\in\R$, and assume for the sake of contradiction that $F$ is not right-continuous at $x$. Then there exists $\delta>0$ such that for any $\eps>0$,  $F(x+\eps)-F(x)>\delta$. Let $\eps>0$, then \begin{equation}\label{eq:F eps inequality}
		F(x+\eps)> F(x)+\delta.
	\end{equation} From \eqref{eq:quantile to cdf def}, \eqref{eq:F eps inequality} and the monotonicity of $Q$ we have that\begin{equation}
		Q(F(x)+\delta)\leq x+\eps.
	\end{equation} As $\eps>0$ was arbitrary, we have that \begin{equation*}
		Q(F(x)+\delta)\leq x,
	\end{equation*} a contradiction of \eqref{eq:quantile to cdf def}. Thus, $F$ defined by \eqref{eq:quantile to cdf def} is right-continuous on $\R$. 
	
	Let $F_1$ and $F_2$ both be the CDF of distinct distributions both with quantile function $Q$. Let $x\in\R$ be such that $F_1(x)>F_2(x)$. Let $p\in(F_2(x),F_1(x))$, then from Lemma \ref{lemma:quantile facts} \ref{eq:inequality iff} \begin{equation*}
		F_1(x)\geq p\, \Leftrightarrow Q(p)\leq x\, \Leftrightarrow F_2(x)\geq p,
	\end{equation*} a contradiction. Hence, $F$ defined by \eqref{eq:quantile to cdf def} is unique.
\end{proof}

The following lemma describes convergence in distribution in terms of quantile functions. 

\begin{lemma}[See van der Vaart \cite{vaart_1998} Lemma 21.2]\label{lemma:quntile convergence}
	Let $X_1,X_2,\dots,$ and $X_\infty$ be real valued random variables with quantile functions $Q_1, Q_2, \dots,$ and $Q_\infty$ respectively. Then $X_n$ converges in distribution to $X_\infty$ if and only if $Q_n(p)$ converges to $Q_\infty(p)$ for every continuity point $p\in[0,1)$ of $Q_\infty$. 
\end{lemma}

\bibliography{fractional_convolution}
\bibliographystyle{abbrv}						
						
					\end{document}